\documentclass[a4paper,11pt,reqno]{amsart} 
\usepackage{amssymb, amsmath, amsthm, mathtools, mathrsfs, bm}
\usepackage{extarrows}
\usepackage[all]{xy}
\usepackage{tikz-cd}
\usepackage{extarrows}
\usepackage{arydshln} 
\usepackage{enumerate}
\usepackage{booktabs,float}
\usepackage{bbm}
\usepackage{xfrac}
\usepackage[square, comma, sort, numbers]{natbib}
\usepackage[colorlinks = true,
            linkcolor = blue,
            urlcolor  = cyan,
            citecolor = black,
            anchorcolor = black]{hyperref}

\theoremstyle{plain}
\newtheorem{theorem}{Theorem}[section]

\newtheorem{lemma}[theorem]{Lemma}

\theoremstyle{definition}

\theoremstyle{remark}
\newtheorem{remark}[theorem]{Remark}

\newcommand{\m}[2]{\ensuremath{M_{2^{#1}}^{#2}}}

\newcommand{\Z}{\ensuremath{\mathbb{Z}}}
\newcommand{\Cs}{\ensuremath{C_{r}^{5,s}}}
\newcommand{\Cse}{\ensuremath{C_{r,\epsilon}^{5,s}}}
\newcommand{\F}{\ensuremath{F^{s}_{r,\epsilon}}}

\title[2-local unstable homotopy groups of indecomposable $\mathbf{A}_3^2$ -complexes]{2-local unstable homotopy groups of indecomposable $\mathbf{A}_3^2$ -complexes}

\author[Z. Zhu]{Zhongjian Zhu}
\address{College of Mathematics and Physics, Wenzhou University, Wenzhou, Zhejiang \rm{325035}, China}
\email{zhuzhongjian@amss.ac.cn}

\author[J. Pan]{Jianzhong Pan}
\address{Hua Loo-Keng Key Mathematical Laboratory, Institute of Mathematics,Academy of Mathematics and Systems Science, Chinese Academy of Sciences,University of Chinese Academy of Sciences, Beijing, {\rm 100190}, China}
\email{pjz@amss.ac.cn}


\begin{document}

\begin{abstract}
In this paper, we calculate the 2-local unstable homotopy groups of  indecomposable $\mathbf{A}_3^2$-complexes. The main technique used is  analysing the homotopy property of $J(X,A)$, defined by B. Gray for a CW-pair $(X,A)$,  which is homotopy equivalent to the homotopy fibre of the pinch map $X\cup CA\rightarrow \Sigma A$.
\end{abstract}
\maketitle

\tableofcontents

\section{Introduction}

\label{intro}

For a suspended  finite CW-complex $X$, if $X\simeq X_1\vee X_2$ and both $X_1$ and $X_2$ are not contractible, then $X$ is called decomposable; otherwise $X$ is called indecomposable. Let $\mathbf{A}_{n}^k$ be the  homotopy category consisting of $(n-1)$-connected finite CW-complexes with dimension less than or equal to $n+k$ $(n\geq k+1)$. The objects of  $\mathbf{A}_{n}^k$ are also called $\mathbf{A}_{n}^k$-complexes. In 1950, S.C.Chang classified the indecomposable homotopy types in $\mathbf{A}_{n}^2 (n\geq 3)$ \cite{RefChang}, that is
\begin{itemize}
	\item [(i)] Spheres:~~ $S^{n}$, $S^{n+1}$, $S^{n+2}$;
	\item [(ii)] Elementary Moore spaces:~~ $M_{p^{r}}^n$ , $M_{p^{r}}^{n+1}$  where $p$ is a prime, $r\in \mathbb{Z}^+$ and $M_{p^r}^{k}$ denotes $M(\Z/p^r, k)$, whose  only nontrivial reduced homology is  $\tilde{H}_{k}(M_{p^r}^{k})=\Z/p^r\Z$;
	\item [(iii)] Elementary Chang complexes:~~ $C_{\eta}^{n+2}$, $C^{n+2,s}$, $C_{r}^{n+2}$, $C_{r}^{n+2,s}$ ( $r,s\in \mathbb{Z}^+$), which are given by the mapping cones of the maps
	$\eta_{n}:S^{n+1}\rightarrow S^{n}$,  $f^s=j^{n+1}_1(2^s\iota_{n+1})+j^n_2\eta_{n}: S^{n+1}\rightarrow S^{n+1}\vee S^n$, $f_r=(\eta_{n}, 2^{r}\iota_{n}):S^{n+1}\vee S^n \rightarrow  S^n$, $f_r^s=(j^{n+1}_1(2^s\iota_{n+1})+j^n_2\eta_{n}, j^n_{2}(2^{r}\iota_{n})):S^{n+1}\vee S^n \rightarrow S^{n+1}\vee S^n$ respectively,
\end{itemize}
where $\mathbb{Z}^+$ denotes the set of positive integers; $\iota_n\in \pi_{n}(S^n)$ is the identity map of $S^n$; $\eta_{2}$ is the Hopf map $S^3\rightarrow S^2$ and  $\eta_{n}=\Sigma^{n-2}\eta_2$ for $n\geq 3$; $j^{n+1}_1$, resp. $j^n_2$ is the inclusion of $S^{n+1}$, resp. $S^n$, into $S^{n+1}\vee S^n$.

The suspension $\Sigma$ gives us sequences of functors $\mathbf{A}_{n}^k\xrightarrow{\Sigma}\mathbf{A}_{n+1}^k$ for all $n\geq k+1$.
The Freudenthal suspension theorem shows that these sequences stabilize in the sense that for $k+1<n$ the functor $\mathbf{A}_{n}^k\xrightarrow{\Sigma}\mathbf{A}_{n+1}^k$  is an  equivalence of additive categories. We point out that for $k+1=n$, the suspension functor $\mathbf{A}_{k+1}^k\xrightarrow{\Sigma}\mathbf{A}_{k+2}^k$ is a full representation equivalence, i.e. it is full, dense and reflects isomorphisms \cite{Drozd}, which implies that $\Sigma$ gives a 1-1 correspondence of homotopy types. Thus we often study   $\mathbf{A}_{k+1}^k$ as a beginning of the study of  $\mathbf{A}_{n}^k$ for $n\geq k+1$. There has been a lot of research on homotopy  of spheres and elementary Moore spaces, but only a few on homotopy  of all indecomposable  $\mathbf{A}_{n}^2$-complexes by taking them as a whole. In the 1950s, P.J.Hilton calculated the $n+1,n+2$-dim homotopy groups of $\mathbf{A}_{n}^2$-complexes \cite{Hilton1950,Hilton1951, Hiltonbook}. In 1985, H.J.Baues calculated the abelian groups $[X,Y]$  and  groups of homotopy equivalences $Aut(X)$  for all indecomposable $\mathbf{A}_{n}^2$-complexes $X$ and $Y$ \cite{BauJHC}. In 2017, the authors obtained the complete wedge decomposition of smash product $X\wedge Y$ for all indecomposable $\mathbf{A}_{n}^2$-complexes $X$ and $Y$ \cite{ZP}, and then as an application, we prove that the stable homotopy groups of elementary Chang complexes $C_{r}^{n+2,r}$ are direct summands of their unstable homotopy groups\cite{ZLP}. In 2020,  we obtained the local hyperbolicity, which is defined by R.Z.Huang and J.Wu to study the asymptotic behavior of the $p$-primary part of the homotopy groups of simply connected finite $p$-local complexes \cite{HuangRuiWu},  of $\mathbf{A}_{n}^2$-complexes by an analysis of decomposition of loop suspension \cite{ZPhyperbolicity}. In recent years, the problem of realisability of groups as self-homotopy equivalences of $\mathbf{A}_{n}^2$-complexes are studied by  C. Costoya, et al.\cite{C.Costoya}. Then  D.$\acute{M}$endez study the problem of realisability of rings as the ring of stable homotopy classes of self-maps of   $\mathbf{A}_{n}^2$-complexes \cite{D.Mendez}.

Calculating the unstable homotopy groups of finite CW-complexes is a fundamental and difficult problem in algebraic topology. A lot of related work \cite{X.G.Liu,Mukai:pi(CPn),Mukai,Toda,WJProjplane,JXYang} has been done on  CW-complexes with the number of cells less than or equal to 2, such as spheres, elementary Moore spaces, projective space and so on. J.Wu calculated the homotpy groups of mod 2 Moore spaces by using the  functorial decomposition \cite{WJProjplane} and recently, J.X.Yang, et al. calculate the homotopy groups of the suspended quaternionic projective plane in \cite{JXYang} by using the relative James construction.
Although calculating the  unstable homotopy groups of a CW-complex with the number of cells greater than 2 will be more complicated, we realize that it is possible to compute homotopy groups of $\mathbf{A}_{n}^2$-complex by similar method after reading their preprint \cite{JXYang}. In this paper, we will calculate the 6 and 7 dimensional unstable homotopy groups of indecomposable $\mathbf{A}_3^2$ -complexes. We should point out that for an  $\mathbf{A}_3^2$ -complex $X$,  the Freudenthal suspension theorem implies that $\pi_{m}(X)$ is in the stable range for $m\leq 4$, and by the calculation of \cite{Hiltonbook} so is $\pi_{5}(X)$ when $X$ is indecomposable. Hence the 6-dimensional homotopy group of an indecomposable $\mathbf{A}_3^2$ -complex (except  $M_{p^r}^4$) is its first unstable homotopy group. As a potential application, the  7-dimensional homotopy group of $\mathbf{A}_3^2$ -complexes may be used to study the classification problem of 2-connected 8-dimensional manifolds $M^8$, since the homotopy class of the attaching map  of its top cell is an element of $\pi_7(X)$, where $X$ is an  $\mathbf{A}_3^2$ -complex.

\begin{theorem} The $6,7$-homotopy groups of all $2$-loacl nontrivial  indecomposable $\mathbf{A}_3^2$-complexes are listed as follows:

	\begin{itemize}
		\item [$(1)$]\begin{align}
			&\pi_{6}(M_{2^r}^3) \cong \left\{
			\begin{array}{ll}
				\Z_4\oplus\Z_2, & \hbox{$r=1$;} \\
				\Z_8\oplus \Z_2\oplus\Z_2, & \hbox{$r=2$;} \\
				\Z_4\oplus\Z_2\oplus\Z_{2^r}, & \hbox{$r\geq 3$;}
			\end{array}
			\right. \nonumber \\
			&\pi_{7}(M_{2^r}^3)\cong \Z_2\oplus \Z_2\oplus (1-\epsilon_r)\Z_4;~~~~~~~~~~~~~~~~~~~~~~~~~
			\nonumber\\
			&\pi_{7}(M_{2^r}^4) \cong  \Z_{2^{min\{2,r-1\}}}\oplus\Z_{2^{r+1}}\oplus\Z_2.\nonumber
		\end{align}
		\item  [$(2)$] \begin{align}
			&\pi_{6}(C_{\eta}^5)\cong \Z_2;\nonumber\\
			&\pi_{6}(C_{r}^{5})\cong
			\Z_2\oplus (1-\epsilon_r)\Z_2\oplus\Z_{2^{r+\epsilon_r}};\nonumber\\
			&\pi_{6}(C^{5,s})\cong \Z_2\oplus\Z_2\oplus \Z_{2^s};\nonumber\\
			&\pi_{6}(\Cs)\cong\Z_2\oplus \Z_2\oplus(1-\epsilon_r)\Z_2\oplus\Z_{2^{min\{r,s\}}}\oplus\Z_{2^{r+\epsilon_r}}.\nonumber
		\end{align}
		\item  [$(3)$]\begin{align}
			&\pi_{7}(C_{\eta}^5)\cong \Z_{(2)};\nonumber\\
			&\pi_{7}(C_{r}^{5})\cong \Z_4\oplus \Z_{2^{r+1}};\nonumber\\
			&\pi_{7}(C^{5,s})\cong  \Z_{2^{min\{s,2\}}}\oplus\Z_{2^{s+2}}; \nonumber\\
			&\pi_{7}(C_{r}^{5,s})\cong
			\Z_{2^{min\{s-\epsilon_r,2\}}}\oplus\Z_{2^{min\{s+1,r+1\}}}\oplus\Z_{2^{s+2}}\oplus \Z_4, \nonumber
		\end{align}
	\end{itemize}
	where $\Z_{(2)}$ denotes the $2$-local integers and $\Z_{k}:=\Z/k\Z$.
\end{theorem}
$\epsilon_r=\left\{
\begin{array}{ll}
	1, & \hbox{$r=1$;} \\
	0, & \hbox{$r\geq 2$}
\end{array}
\right.$ in the Theorem and  we also set $\epsilon_{\infty}=0$ when $r=\infty$ is allowed in the following text.

The proof of the first statement of the Theorem is given in Section 3 and the  remaining proofs are given in Section \ref{subsec:pi6 C} and \ref{subsec:pi7 C} respectively.

\section{Some notations and  lemmas}

In this paper,  all spaces and maps are in the category of pointed CW-complexes and maps (i.e. continuous functions) preserving basepoint. And we always use $*$ and $0$ to denote the basepoints and the constant maps mapping to the basepoints respectively. We denote $A\hookrightarrow X$  as an inclusion map.

Let  $(X,A)$ be a pair of  spaces with base point $*\in A$, and suppose that $A$ is closed in $X$. In \cite{Gray}, B.Gray constructed a space $(X,A)_{\infty}$ analogous to the James construction, which is denoted by us as $J(X,A)$ to parallel with the the absolute James construction $J(X)$. In fact, $J(X,A)$ is the subspace of $J(X)$ of words for which letters after the first are in $A$.  Especially,  $J(X,X)=J(X)$. As parallel with the familiar symbol $J_{r}(X)$ which is the $r$-th filtration of $J(X)$, we denote the  $r$-th filtration of $J(X,A)$ by  $J_{n}(X,A):=J(X,A)\cap J_{r}(X)$, which is denoted by Gray as  $(X, A)_r$ in  \cite{Gray}.

For example, $J_{1}(X,A)=X$, $J_{2}(X,A)=(X\times A)/((a,\ast)\thicksim (\ast,a))$ for each $a\in A$. In fact there is a  pushout diagram for $r\geq 2$:
$$ \footnotesize{\xymatrix{
		X\times A^{n-1} \ar[r]^{\Pi_{r}} &  J_{r}(X, A) \\
		F \ar@{^{(}->}[u]\ar[r] & J_{r-1}(X, A)\ar@{^{(}->}[u]^{I_{r}} } }$$
where $F\subset X\times  A^{n-1}$ is the ``fat wedge" consisting of those points in which one or more coordinates is the base-point; $\Pi_{r}$ and $I_{r}$  are the projection and  the inclusion respectively and both of them are natural.

\begin{remark}$J_{n}(X,A)/J_{n-1}(X,A)$ is naturally homeomorphic to $(X\times A^{n-1})/F=X\wedge A^{\wedge (n-1)}$.
\end{remark}
It is well known that there is a natural weak homotopy equivalence  $\omega:J(X)\rightarrow \Omega\Sigma X$, which is a homotopy equivalence when $X$ is a finite CW-complex, and satisfies
$\footnotesize{\xymatrix{
		&X\ar@/_0.5pc/[rr]_{\Omega\Sigma}\ar@{^{(}->}[r]& J(X) \ar[r]^{\omega}& \Omega\Sigma X } }$,
where $X\xrightarrow{\Omega\Sigma} \Omega\Sigma X$ is the inclusion $x\mapsto \psi$ where $\psi: S^1\rightarrow S^1\wedge X, t\mapsto  t\wedge x .$

Let $X\xrightarrow{f}Y$ be a map.  We always use $C_f$, $F_f$ and  $M_f$ to denote the maping cone ( or say, cofibre ),  homotopy fibre and  mapping cylinder of $f$, $C_f\xrightarrow{p} \Sigma X$ the pinch map and $\Omega\Sigma X\xrightarrow{\partial}F_p\rightarrow C_f\xrightarrow{p}\Sigma X$ the homotopy fibration sequence induced by $p$ respectively. We get the relative James construction $J(M_f,X)$  (resp. $r$-th relative James construction $J_{r}(M_f,X)$ ) for the pair $(M_f,X)$.

\begin{lemma}\label{lemma Gray} Let $X\xrightarrow{f} Y$ be a map.  Then we have
	\begin{itemize}
		\item [(i)]  $F_{p}\simeq J(M_f,X)$;
		\item [(ii)] $\Sigma J(M_f,X)\simeq \bigvee_{k\geq 0}(\Sigma Y\wedge X^{\wedge k})$;$\Sigma J_{k}(M_f,X)\simeq \bigvee_{i= 0}^{k-1}(\Sigma Y\wedge X^{\wedge i})$;
		\item [(iii)] If $Y=\Sigma Y'$, $X=\Sigma X'$, then  $J_2(M_f,X)\simeq Y\cup_{\gamma}C(Y\wedge X')$, where $\gamma=[id_Y, f]$ is the generalized Whitehead product.
	\end{itemize}
\end{lemma}
\begin{proof}
	The lemma follows from the Theorems of \cite{Gray} for $(M_f,X)$.
\end{proof}

Denote both the inclusion  $Y\hookrightarrow J_{2}(M_f,X)$ and the composition of the inclusions
$Y\hookrightarrow J_{2}(M_f,X)\hookrightarrow J(M_f,X) \simeq  F_{p}$ by $j_{p}$ without ambiguous.

\begin{lemma}\label{J(X,A)toJ(X',A')}
	Suppose the left  diagram  is commutative\\
	$ \footnotesize{\xymatrix{
			X\ar[d]^{\mu'}\ar[r]^-{f} & Y\ar[d]^{\mu}\\
			X'\ar[r]^-{f'} & Y'}}$;~~~~~~$ \footnotesize{\xymatrix{
			F_p\simeq  J(M_f,X)\ar[d]^{J(\widehat{\mu},\mu')}\ar[r]&M_f/X\simeq C_f  \ar[d]^{\bar{\mu}}\ar[r]& \Sigma X\ar[d]^{\Sigma\mu'}\\
			F_{p'}\simeq J(M_{f'},X')\ar[r]& M_{f'}/X'\simeq C_{f'} \ar[r]&\Sigma X'}}$
	
	then it induces the right commutative diagrams on fibrations, where $\widehat{\mu}$ satisfies
	
	$ \footnotesize{\xymatrix{
			&Y\ar@/^0.5pc/[rrr]^{\mu}\ar@{^{(}->}[r]_{\simeq}& M_f \ar[r]_{\widehat{\mu}}& M_{f'}\ar[r]_{\simeq}& Y'} }$.  Let
	
	$M_f=J_{1}(M_f,X)\xrightarrow{J(\widehat{\mu},\mu')|_{M_f}=J_{1}(\widehat{\mu},\mu')=\widehat{\mu}} J_{1}(M_{f'},X')=M_{f'}$,
	
	$J_{2}(M_f,X)\xrightarrow{J(\widehat{\mu},\mu')|_{J_{2}(M_f,X)}=J_{2}(\widehat{\mu},\mu')} J_{2}(M_{f'},X')$,
	
	then we have the following commutative diagram
	$$ \footnotesize{\xymatrix{
			Y\wedge X \ar[d]_{\simeq} \ar[r]^{\mu\wedge\mu'} & Y'\wedge X' \ar[d]_{\simeq}\\
			J_{2}(M_f,X)/J_{1}(M_f,X)\ar[r]^-{\overline{J_{2}(\widehat{\mu},\mu')}} & J_{2}(M_{f'},X')/J_{1}(M_{f'},X')} }$$
\end{lemma}

\begin{proof}
	The above lemma is easily obtained from \cite{Gray}.
\end{proof}

The following Lemma \ref{partial calculate1} to Lemma \ref{relate partial to H2} come  from \cite{JXYang} in original or generalized form.
\begin{lemma}\label{partial calculate1}
	Let $X\xrightarrow{f}Y$ be a map. Then the following diagram is homotopy commutative
	$$ \footnotesize{\xymatrix{
			X\ar[d]_{\Omega\Sigma } \ar[r]^{f} & Y \ar@{^{(}->}[d]\\
			\Omega\Sigma X \ar[r]^{\partial} & F_p } }$$
\end{lemma}

\begin{proof}
	We have the following homotopy-commutative diagram
	$$ \footnotesize{\xymatrix{
			&X\ar[dl]_{f}\ar@/^1pc/[rr]^{\Omega\Sigma}\ar@{_{(}->}[d]_{i} \ar@{^{(}->}[r]& J(X) \ar[r]^{\omega}\ar[d]_{J(i,id_X)}& \Omega\Sigma X\ar[d]^{\partial}\\
			Y\ar@{^{(}->}[r]^{\simeq }&M_f\ar@{^{(}->}[r] & J(M_f, X)\ar[r]^{\simeq }& F_p } }$$
	where the middle homotopy-commutative square comes from the naturality of the relative James construction and the  right homotopy-commutative square comes from Lemma 4.1 of \cite{Gray}.
	Thus the Lemma \ref{partial calculate1} is obtained.
	
\end{proof}

\begin{lemma}\label{Diagram of phi}
	Let $X\xrightarrow{f} Y \xrightarrow{i}C_f \xrightarrow{p} \Sigma X\xrightarrow{-\Sigma f} \Sigma Y $be a cofibration sequence. Then there is a homotopy commutative diagram with rows fibration sequences:
	$$ \footnotesize{\xymatrix{
			\Omega\Sigma X\ar@{=}[d]\ar[r]^-{\partial} &J(M_f,X)\ar[d]^{\phi}\ar[r] & C_f\ar@{_{(}->}[d]\ar[r]^{p}& \Sigma X \ar@{=}[d]& \\
			\Omega\Sigma X\ar[r]^{\Omega(-\Sigma f)} & \Omega\Sigma Y \ar[r] & J(C_{f},Y)\ar[r]&\Sigma X \ar[r]^-{-\Sigma f}&\Sigma Y}}$$
\end{lemma}
\begin{proof}
	As pointed out in  the proof of Lemma 4.1.of \cite{Gray}, there is a natural inclusion $C_f\hookrightarrow J(C_f,Y)$ lifting the inclusion
	$ \xymatrix{
		C_{f}\ar@/_1pc/[rr]_{p}\ar@{^{(}->}[r] &C_f\cup_{i}CY\ar[r]^{\simeq } &\Sigma X\\ }$
	The homotopy commutativity of the right square implies that there exists a map $J(M_f,X)\xrightarrow{\phi} \Omega\Sigma Y$ such that the left and the middle squares are homotopy commutative.
\end{proof}

Similar to the James-Hopf invariant, we define the $n$-th relative James-Hopf invariant
$$J(X,A)\xrightarrow{H_{n}} J(X\wedge A^{\wedge(n-1)}),~x_1x_2\dots x_t\mapsto\prod\limits_{1\leq x_1<x_2<\dots<x_t\leq n}(x_{i_{1}}\wedge x_{i_{2}}\wedge\dots\wedge x_{i_{t}}  )$$ which are natural for pairs.
$H_{n}(x_1x_2\dots x_n)=x_1\wedge x_2\wedge \dots\wedge x_n$ implies the following lemma

\begin{lemma}\label{lem for H2}
	Let $X\xrightarrow{f}Y$ be a map. Then the following diagram is homotopy commutative
	$$ \footnotesize{\xymatrix{
			J(M_f,X)\ar[d]_{H_{n}}& \ar@{_{(}->}[l]J_{n}(M_f,X) \ar[d]^-{pinch} \\
			J(M_f\wedge X^{\wedge (n-1)} )&M_f\wedge X^{\wedge (n-1)}=J_{n}(M_f,X)/ J_{n-1}(M_f,X)\ar@{_{(}->}[l]} }$$
\end{lemma}

\begin{remark}
	
	By abuse of notion, $H_2$ also denotes the composition of the maps $\Omega\Sigma X \xrightarrow{\simeq} J(X)=J(X,X)\xrightarrow{H_2} J(X\wedge X)\xrightarrow{\simeq} \Omega\Sigma (X\wedge X)$, where $X$ is a CW-complex and let $H_2': F_p\simeq J(M_f,X)\xrightarrow{H_2} J(M_f\wedge X)\xrightarrow{\simeq} J(Y\wedge X)\simeq \Omega\Sigma (Y\wedge X)$.
\end{remark}

\begin{lemma}\label{relate partial to H2}
	Let $X\xrightarrow{f}Y$ be a map. Then the following diagram is homotopy commutative
	$$ \footnotesize{\xymatrix{
			\Omega\Sigma X\ar[d]_{H_2 } \ar[r]^{\partial} & F_p \ar[d]^{H'_2}\\
			\Omega\Sigma (X\wedge X) \ar[r]^-{\Omega\Sigma(f\wedge id_{A})} & \Omega\Sigma (Y\wedge X)~~. } }$$
\end{lemma}
\begin{proof} By the Lemma 4.1 of \cite{Gray} and the naturality of the  2nd relative James-Hopf invariant, we have the following homotopy commutative diagram
	$$ \footnotesize{\xymatrix{
			\Omega\Sigma X\ar@{.>}[d]_{H_2}\ar@/^1pc/[rrr]^{\partial} & J(X)=J(X,X)\ar[l]^-{\simeq}\ar[r]_-{J(i)}\ar[d]_{H_2} & J(M_f,X)\ar[r]_-{\simeq} \ar[d]_{H_2}& F_p\ar@{.>}[d]_{H'_2}\\
			\Omega\Sigma(X\wedge X)\ar@/_1pc/[rrr]_{\Omega\Sigma(f\wedge id_{X})} &J\ar[l]_-{\simeq} (X\wedge X)\ar[r]^-{J(i\wedge id_{X})} & J(M_f\wedge X) \simeq \Omega\Sigma(M_f\wedge X) \ar[r]^-{\simeq} & \Omega\Sigma(Y\wedge X) } }$$
	We complete the proof.
\end{proof}

The following lemma comes from \cite{Cohen}
\begin{lemma}\label{stable exact seq}
	If $X\xrightarrow{f} Y$ is a map, $X$ is $n-1$ connected, $C_f$ is $m-1$ connected, the dimension of $W$ is less than or equal to $m+n-2$, then we have the exact sequence $$[W,X]\xrightarrow{f_{\ast}} [W,Y] \rightarrow [W,C_{f}].$$
\end{lemma}

\begin{lemma}\label{Lem:Spliting}
	Let $p$ be a prime  and suppose that there is a commutative diagrams of short exact sequences of $p$-torsion  abelian groups with $s<r$
	
	$\small{\xymatrix{
			0 \ar[r]&B_1\ar[r]^-{i_2} & A_1 \ar[r]^{p_2}&\Z_{p^r}\ar[r]& 0\\
			0\ar[r]&B\ar[u]\ar[r]^-{i_1}& A\ar[u]^{f}\ar[r]^{p_1} &\Z_{p^s} \ar@_{(->}[u]_{J}\ar[r]& 0
	} }. \label{Diagram pi6(M2^1)to pi6(M2^r)}$
	
	If the characteristic  $ch(B_1)\leq p^s$ and  the bottom short exact sequence is split
	then so is the top.
\end{lemma}
\begin{proof}
	It follows from an easy diagram chasing argument.
\end{proof}

The following generators of homotopy groups of spheres after localization at 2  come from \cite{Toda}.
$\iota_n=[id]\in\pi_n(S^n)$; $\pi_{3}(S^2)=\Z_{(2)}\{\eta_2\}$;  $\pi_{n+1}(S^n)=\Z_2\{\eta_n\} (n\geq 3)$;
$\pi_{n+2}(S^n)=\Z_2\{\eta_n\eta_{n+1}\} (n\geq 3)$; $\pi_{6}(S^3)=\Z_4\{\nu'\}$;  $\pi_{7}(S^4)=\Z_4\{\Sigma\nu'\}\oplus \Z_{(2)}\{\nu_4\}$; $\pi_{n+3}(S^n)=\Z_8\{\nu_n\} (n\geq 5)$; $\pi_{7}(S^3)=\Z_2\{\nu'\eta_6\}$; $\pi_{8}(S^4)=\Z_2\{\Sigma \nu'\eta_7\}\oplus\Z_2\{\nu_4\eta_7\}$.

Throughout the paper, we will not distinguish a map and its homotopy class in many cases.

In the following  all spaces  are  2-local.  $2^r=0$ is allowed, in this case we denote $r=\infty$, i.e,  $2^{\infty}=0$,  $\Z_{2^{\infty}}=\Z_{0}=\Z_{(2)}$ (after 2-localization), $\Z_1=0$ (trivial group),  $min\{k,\infty\}=k$ for some integer $k$.

\section{Elementary Moore spaces}

In this section we calculate  $\pi_{n}(M_{2^r}^{3}) (n=6,7)$ and $\pi_{7}(M_{2^r}^{4})$. For $r=1,2,3$,  many homotopy groups of these Moore spaces have been calculated by J. Wu, J. Mukai, T. Shinpo, and X.G.Liu  in \cite{WJ Proj plane},\cite{Mukai}, and \cite{X.G.Liu} respectively.

There is a canonical cofibration sequence
\begin{align}
	S^k\xrightarrow{2^{r}\iota_k} S^{k}\xrightarrow{i_{k}} M_{2^r}^{k}\xrightarrow{p_{k}}  S^{k+1} \label{Cofiberation for Mr^k}
\end{align}
where $M_{2^{\infty}}^{k}=M_{0}^{k}=S^{k}\vee S^{k+1}$.

Let $\Omega S^{k+1}\xrightarrow{\partial}  F_{p_{k}}\rightarrow M_{2^r}^{k}\xrightarrow{p_{k}}  S^{k+1}$ be the homotpy fibration sequence.
By Lemma \ref{lemma Gray} we get $\Sigma F_{p_{k}}\simeq S^{k+1}\vee S^{2k+1}\vee\dots$.

\subsection{Calculating $\pi_{6}(M_{2^r}^{3})$}
For $M_{2^r}^{3}$, the 8-skeleton  $Sk_{8}(F_{p_{3}})\simeq S^3\bigcup_{\gamma}CS^5$ with $\Sigma\gamma=0$ and isomorphism $ [S^5, S^3]\xrightarrow{\Sigma }[S^6, S^4]$ implies that $\gamma=0$. Thus the 8-skeleton  $Sk_{8}(F_{p_{3}})\simeq S^3\vee S^6$. By  Lemma \ref{partial calculate1}, we have exact sequence with commutative squares
\begin{align} \small{\xymatrix{
			\pi_{7}( S^4)\ar[r]^-{ \partial_{6\ast}}&\pi_{6}(F_{p_3})\ar[r]&\pi_{6}(M_{2^r}^{3})\ar[r]^-{q_{r\ast}}&\pi_{6}(S^4)\ar[r]^-{ \partial_{5\ast}}& \pi_{5}(F_{p_3})\\
			\pi_{6}(S^3)\ar[r]^-{(2^r\iota_{3})_{\ast}}\ar[u]_{\Sigma } &\pi_{6}(S^3)\ar@{^{(}->}[u]_{\cong j_{p_3\ast}}&&\pi_{5}( S^3)\ar[r]^-{(2^r\iota_{3})_{\ast}=0}\ar[u]_{\cong\Sigma } &\pi_{5}(S^3)\ar@{^{(}->}[u]_{j_{p_3\ast}} }}  \label{exact pi6(M^3)}
\end{align}
$\pi_{5}(F_{p_3})=\Z_2\{j_{p_{3}}\eta_3\eta_4\}$;
$\pi_{6}(F_{p_3})=\Z_4\{j_{p_3}\nu'\}\oplus\Z_{(2)}\{j^6_{p_3}\}$,
where  $j^6_{p_{3}}: S^6\hookrightarrow Sk_8(F_{p_{3}})$ is the inclusions of the wedge summand $ S^6$.

By the right commutative square of (\ref{exact pi6(M^3)}), we get $Ker \partial_{5\ast}\cong \Z_2\{\eta_4\eta_5\}$.

Next calculate $Coker \partial_{6\ast}$ in (\ref{exact pi6(M^3)}).

By Lemma 4.5 of \cite{Toda}, $(k\iota_{n})\alpha= k\alpha$ for $\alpha\in \pi_{n}(S^3)$, then by the left commutative square of (\ref{exact pi6(M^3)})
\begin{align}
	\partial_{6\ast}(\Sigma \nu')=2^rj_{p_3}\nu'.
	\label{partial6 on Sigma nu'}
\end{align}

Lemma \ref{lem for H2}, Lemma \ref{relate partial to H2} for map $S^3\xrightarrow{2^r\iota_3}S^3$ give the following commutative diagram
\begin{align}
	\footnotesize{\xymatrix{
			\pi_7( S^4)\ar[d]_{H_2} \ar[r]^-{\partial_{6\ast}} & \pi_{6}(F_{p_3})\ar[d]_{H'_2 }&\pi_{6}(S^3\vee S^6)\!\ar@{_{(}->}[l]^{\cong}\ar[r]^-{Proj}&\!\!\pi_{6}(S^6)\ar[d]_{\Sigma \cong}\\
			\pi_{6}(\Omega\Sigma S^3\wedge S^3) ~~~\ar[r]^-{(\Omega\Sigma2^r\wedge \iota_3)_{\ast}} & \pi_{6}(\Omega\Sigma S^3\wedge S^3)\ar@{=}[rr]&&\! \pi_{7}(S^7) } }  \label{diagram H2,H2'M^3}
\end{align}
By the right commutative square, $H'_{2}(j^6_{p_3})=\iota_7$ for $j^6_{p_3}\in \pi_6(F_{p_3})$.
$H_{2}(\nu_{4})=\iota_7$ by Lemma 5.4 of \cite{Toda}.

Thus from the left commutative square of  (\ref{diagram H2,H2'M^3}), we get
\begin{align}
	\partial_{6\ast}(\nu_4)=yj_{p_3}\nu'+2^rj^6_{p_3} ~~\text{for some}~y\in\Z_4.
	\label{partial6 on nu4}
\end{align}
From Lemma \ref{Diagram of phi}, we have the following two (homotopy) commutative diagrams
\begin{align} \footnotesize{\xymatrix{
			& S^3 \ar@{_{(}->}[d]_{j_{p_3}}\ar@/^1pc/[dd]\\
			\Omega S^4\ar@{=}[d] \ar[r]&F_{p_3}\ar[r]\ar[d]_{\phi}& M_{2^r}^3\ar@{_{(}->}[d]\\
			\Omega S^4 \ar[r]^-{\Omega(\!-\!2^r\!\iota_4)}&\Omega S^4\ar[r] & J(M_{2^r}^3,S^3),
	} }~~ \footnotesize{\xymatrix{
			&\Z_4\{j_{p_3}\nu'\}\oplus\Z_{(2)}\{j^6_{p_3}\}\ar@{=}[d]\\
			\pi_{7}( S^4)\ar[rd]_-{P_1(-2^r\iota_4)_\ast} \ar[r]^-{{\partial_{6\ast}}} & \!\pi_{6}(F_{p_3})\ar[d]_{P_1\phi_{\ast} }\\
			& \Z_4\{\Sigma \nu'\} } }\label{Diagram phi j}
\end{align}
$P_1:\pi_{7}( S^4)=\Z_4\{\Sigma \nu'\}\oplus\Z_{(2)}\{\nu_{4}\}\rightarrow \Z_4\{\Sigma \nu'\}$ is the canonical projection.

By comparing the Homology $H_{3}(-;\Z)$, we get
\begin{align}
	\phi j_{p_3}\simeq h\Omega\Sigma: S^3\rightarrow \Omega S^4 =\Omega\Sigma S^3,  ~ h ~\text{is odd integer}. \label{phi j=dOmegaSigam}
\end{align}
$(-2^r\iota_4)_{\ast}( \nu_4)=2^{2r}\nu_4-2^{r-1}(2^r+1)\Sigma \nu'$ by Lemma \ref{lem[lota4,iota4]H(nu4)}.
\begin{align}
	&P_1\phi_{\ast}\partial_{6\ast}( \nu_4)=P_1(y\phi_{\ast}(j_{p_3}\nu')+ 2^r\phi_{\ast}(j^6_{p_3}))=hy\Sigma \nu'+2^rP_1\phi_{\ast}(j^6_{p_3})  ~~~~ (\text{By (\ref{partial6 on nu4}}))\nonumber\\
	&=(hy+2^rt)\Sigma \nu' ~~~~~~(\text{for some odd integer } t). \nonumber\\
	&P_1(-2^r\iota_4)_\ast(\nu_4)=P_1(2^{2r}\nu_4-2^{r-1}(2^r+1)\Sigma \nu')=-2^{r-1}(2^r+1)\Sigma \nu'.\nonumber
\end{align}
From  (\ref{Diagram phi j}), we get $hy\Sigma \nu'+2^rP_1\phi_{\ast}(j^6_{p_3})=-2^{r-1}(2^r+1)\Sigma \nu'$, thus
\begin{align}
	\!\!\!\!\!\!\footnotesize{\begin{tabular}{r|c|c|c|}
			\cline{2-4}
			& $ r=1$&$r=2$&$\infty\geq r\geq 3$ \\
			\cline{2-4}
			$\Z_4\ni y=$  &  $  \pm  1$ & $2$ &$0$\\
			\cline{2-4}
	\end{tabular}}~  \label{equation y}
\end{align}

From (\ref{partial6 on Sigma nu'}), (\ref{partial6 on nu4}), (\ref{equation y})
\begin{align}
	Coker \partial_{6\ast}\cong \frac{\Z_4\oplus \Z_{(2)}}{ \left \langle  (2^r, 0), (y, 2^r) \right \rangle}\cong \left\{
	\begin{array}{ll}
		\Z_4, & \hbox{$r=1$;} \\
		\Z_8\oplus \Z_2, & \hbox{$r=2$;}\\
		\Z_4\oplus \Z_{2^r}, & \hbox{$ r\geq 3$.}
	\end{array}
	\right.
	\label{Cokpartial6 M3}
\end{align}

\begin{align}
	\small{\xymatrix{
			0\ar[r]&Coker\partial_{6\ast} \ar[r]& \pi_{6}(M_{2^r}^3) \ar[r]^-{p_{3\ast}}&Ker\partial_{5\ast}\cong\Z_2\{\eta_4\eta_5\}\ar[r]& 0
	} }.  \label{exact pi6(Mr3)}
\end{align}
The above short exact sequence splits for $r=1$, since $\pi_{6}(M_{2}^3)\cong \Z_4\oplus\Z_2$ by \cite{WJProjplane}, i.e., there is an element  $\varsigma_1\in \pi_{6}(M_{2}^3)$ with order 2 such that $p_{3\ast}(\varsigma_1)=\eta_4\eta_5$. Hence by the Lemma 2.5 of \cite{Mukai}, for $r>1$, there is also an element  $\varsigma_r\in \pi_{6}(M_{2^{r}}^3)$ with order 2 such that $p_{3\ast}(\varsigma_r)=\eta_4\eta_5$. Thus short exact sequence (\ref{exact pi6(Mr3)}) splits for $r>1$.
\begin{align}
	\text{So}~~~~~~~~~ \pi_{6}(M_{2^r}^3)\cong \left\{
	\begin{array}{ll}
		\Z_4\oplus\Z_2, & \hbox{$r=1$;} \\
		\Z_8\oplus \Z_2\oplus\Z_2, & \hbox{$r=2$;} \\
		\Z_4\oplus\Z_2\oplus\Z_{2^r}, & \hbox{$r\geq 3$.}
	\end{array}
	\right. \nonumber
\end{align}

\subsection{Calculating $\pi_{7}(M_{2^r}^3)$}
Consider the following  diagram
\begin{align} \small{\xymatrix{
			\pi_{8}( S^4)\ar[r]^-{ \partial_{7\ast}}&\pi_{7}(F_{p_3})\ar[r]&\pi_{7}(M_{2^r}^{3})\ar[r]^-{q_{r\ast}}&\pi_{7}(S^4)\ar[r]^-{ \partial_{6\ast}}& \pi_{6}(F_{p_3})\\
			\pi_{7}(S^3)\ar[r]^-{(2^r\iota_{3})_{\ast}}\ar[u]_{\Sigma } &\pi_{7}(S^3)\ar@{^{(}->}[u]_{\cong j_{p_3\ast}}&&& }}  \label{exact pi7(M^3)}
\end{align}
where the first row is exact sequence, and the left square follows from Lemma \ref{partial calculate1}.
$\pi_{7}(F_{p_3})=\Z_2\{j_{p_3}\nu'\eta_6\}\oplus\Z_2\{j_{p_3}^6\eta_6\}$.\\
From (\ref{partial6 on Sigma nu'}), (\ref{partial6 on nu4}),
~~~~~~~~~~~$Ker\partial_{6\ast}=\left\{
\begin{array}{ll}
	\Z_2\{2\Sigma \nu'\}  & \hbox{$r=1$;} \\
	\Z_4\{\Sigma \nu'\} , & \hbox{$r\geq 2$ .}
\end{array}
\right.$

\begin{align}
	\partial_{7\ast}(\Sigma \nu'\eta_7)=j_{p_3}(2^r\iota_{3}) \nu'\eta_6=0
	\label{partial7 Sigma nu'eta7}
\end{align}
Assume  $\partial_{7\ast}(\nu_4\eta_7)=aj_{p_3}\nu'\eta_6+bj_{p_3}^6\eta_6$ with $a,b\in \Z_2$. By Lemma \ref{lem for H2} and Lemma \ref{relate partial to H2}, we get the following commutative diagrams

\begin{align}
	\footnotesize{\xymatrix{
			\pi_8( S^4)\ar[d]_{H_2} \ar[r]^-{\partial_{7\ast}} & \pi_{7}(F_{p_3})\ar[d]_{H'_2 }&\pi_{7}(S^3\vee S^6)\!\ar@{_{(}->}[l]^{\cong}\ar[r]^-{Proj}&\!\!\pi_{7}(S^6)\ar[d]_{\Sigma \cong}\\
			\pi_{7}(\Omega\Sigma S^3\wedge S^3) ~~~\ar[r]^-{(\Omega\Sigma2^r\wedge \iota_3)_{\ast}} & \pi_{7}(\Omega\Sigma S^3\wedge S^3)\ar@{=}[rr]&&\! \pi_{8}(S^7) } } \nonumber
\end{align}
\begin{align}
	&0=(\Omega\Sigma2^r\wedge \iota_3)_{\ast}H_2(\nu_4\eta_7)=H'_2\partial_{7\ast}(\nu_4\eta_7)=H'_2(aj_{p_3}\nu'\eta_6+bj_{p_3}^6\eta_6)=b\eta_7, \nonumber
\end{align}
which implies that  $b=0$.

Diagram (\ref{Diagram phi j}) induces the following commutative diagram
$$ \footnotesize{\xymatrix{
		\pi_{8}( S^4)\ar[dr]_-{(-2^r\iota_4)_\ast} \ar[r]^-{{\partial_{7\ast}}} & \!\pi_{7}(F_{p_3})=\Z_2\{j_{p_3}\nu'\eta_6\}\oplus\Z_2\{j^6_{p_3}\eta_6\}\ar[d]_{\phi_{\ast} }\\
		& \pi_{8}(S^4)=\Z_2\{\Sigma \nu'\eta_7\}\oplus\Z_2\{\nu_{4}\eta_7\} \xrightarrow{P_1}\Z_2\{\Sigma \nu'\eta_7\}} }$$
where $P_1:\Z_2\{\Sigma \nu'\eta_7\}\oplus\Z_2\{\nu_{4}\eta_7\}\rightarrow \Z_2\{\Sigma \nu'\eta_7\}$ is the canonical projection.
\begin{align}
	&P_1(-2^r\iota_4)_{\ast}(\nu_4\eta_7)=P_1((2^{2r}\nu_4-2^{r-1}(2^r+1)\Sigma \nu')\eta_7)
	=\epsilon_r \Sigma \nu'\eta_7.  \nonumber\\
	&=P_1\phi_\ast\partial_{7\ast}(\nu_4\eta_7)=P_1\phi_\ast(aj_{p_3}\nu'\eta_6)=a\Sigma\nu'\eta_7,\nonumber
\end{align}
Hence $a=\epsilon_r, \infty\geq r\geq 1$.                                              Thus
\begin{align}
	&\partial_{7\ast}(\nu_4\eta_7)=\epsilon_rj_{p_3}\nu'\eta_6, \label{partial7nu4eta7}
\end{align}
The diagram (3) of \cite{Mukai} induces the following commutative diagram for $r>1$
\begin{align}
	\! \!\! \! \!\small{\xymatrix{
			0\ar[r]&\Z_2\{j_{p_3}\nu'\eta_6\}\!\oplus\!\Z_2\{j_{p_3}^6\eta_6\}  \ar[r]& \pi_{7}(M_{2^r}^3) \ar[r]&\Z_4\{\Sigma \nu'\}\ar[r]& 0\\
			0\ar[r]&\Z_2\{j_{p_3}^6\eta_6\}\ar[u]\ar[r]& \pi_{7}(M_{2}^3)\ar[u]_{c_{4\ast}}\ar[r] &\Z_2\{2\Sigma \nu'\} \ar@^{(->}[u]_{\times 2}\ar[r]& 0
	} }. \label{Diagram pi6(M2^1)to pi6(M2^r)}
\end{align}

By \cite{WJProjplane}, $\pi_{7}(M_{2}^3)\cong \Z_2\oplus \Z_2$, hence from  Lemma \ref{Lem:Spliting}, the top short exact sequence splits. Thus
$\pi_{7}(M_{2^r}^3)\cong \Z_2\oplus \Z_2\oplus (1-\epsilon_r)\Z_4$ for $\infty\geq r\geq 1$.

\subsection{Calculating $\pi_{7}(M_{2^r}^{4})$}
For $M_{2^r}^{4}$,  by (iii) of Lemma \ref{lemma Gray}, the 8-skeleton  $Sk_{8}(F_{p_{4}})\simeq S^4\bigcup_{\gamma=[\iota_4, 2^r\iota_4]}C(S^4\wedge S^3)=S^4\bigcup_{2^r[\iota_4, \iota_4]}C S^7$. Then the cofibration sequence $S^7\xrightarrow{\gamma}S^4\stackrel{j_{p_4}}\hookrightarrow F_{p_4}\rightarrow S^8$ induces the following exact sequence by Lemma \ref{stable exact seq}:
\begin{align}
	\Z_{(2)}\{\iota_7\}=\pi_{7}(S^7)\xrightarrow{\gamma_{\ast}}\pi_{7}(S^4)=\Z_4\{\Sigma \nu'\}\oplus\Z_{(2)}\{\nu_{4}\}\xrightarrow{j_{p_4 \ast}}\pi_{7}(F_{p_4})\rightarrow 0
	\nonumber
\end{align}
with $\gamma_\ast(\iota_7)=(2^r[\iota_4,\iota_4])\iota_7=2^{r+1}\nu_4-2^r\Sigma\nu'$.

Consider the following exact sequence with commutative squares
\begin{align} \small{\xymatrix{
			\pi_{8}( S^5)\ar[r]^-{ \partial_{7\ast}}&\pi_{7}(F_{p_4})\ar[r]&\pi_{7}(M_{2^r}^{3})\ar[r]&\pi_{7}(S^5)\ar[r]^-{ \partial_{6\ast}}& \pi_{6}(F_{p_4})\\
			\pi_{7}(S^4)\ar[r]^-{(2^r\iota_{4})_{\ast}}\ar[u]_{\Sigma } &\pi_{7}(S^4)\ar@{->>}[u]_{ j_{p_4\ast}}&&\pi_{6}( S^4)\ar[r]^-{(2^r\iota_{4})_{\ast}=0}\ar[u]_{\cong\Sigma } &\pi_{6}(S^4)\ar@{^{(}->}[u]_{j_{p_4\ast}} }}   \label{exact pi7(M^4)}
\end{align}

The right commutative square in (\ref{exact pi7(M^4)}) implies $Ker\partial_6\ast\cong \Z_2$.

By the left commutative square in (\ref{exact pi7(M^4)}), we get
\begin{align}
	\partial_{7\ast}(\nu_5)&=\partial_{7\ast}(\Sigma\nu_4)=j_{p_{4}\ast}[(2^r\iota_4)\nu_4]=2^{2r}j_{p_4}\nu_4-2^{r-1}(2^r-1)j_{p_4}\Sigma\nu'.
	\label{partial7 on  nu5}
\end{align}
\begin{align}
	&Coker\partial_{7\ast}=\frac{\Z_4\{j_{p_4}\Sigma \nu'\}\oplus \Z_{(2)}\{j_{p_4}\nu_4\}}{ \left \langle  2^{r+1}j_{p_4}\nu_4-2^rj_{p_4}\Sigma\nu',  2^{2r}j_{p_4}\nu_4-2^{r-1}(2^r-1)j_{p_4}\Sigma\nu'  \right \rangle}\nonumber\\
	&\cong \frac{\Z_{(2)}\{a, b\}}{\left \langle  2^{r+1}b-2^ra,2^{2r}b-2^{r-1}(2^r-1)a, 4a  \right \rangle }=\frac{\Z_{(2)}\{a, b\}}{\left \langle  2^{r+1}b,2^{min\{r-1,2\}}a  \right \rangle }. \nonumber\\
	&= \Z_{2^{min\{2,r-1\}}}\{a\}\oplus\Z_{2^{r+1}}\{b\}.\nonumber
\end{align}
Since $\pi_{7}(M_{2}^4)\cong \Z_2\oplus\Z_4$ from Theorem 5.10.of \cite{WJProjplane},  the same argument in dealing with diagram (\ref{Diagram pi6(M2^1)to pi6(M2^r)}) implies that
\begin{align}
	\pi_{7}(M_{2^r}^4) \cong  Coker \partial_{7\ast}\oplus Ker\partial_{6\ast}\cong
	\Z_{2^{min\{2,r-1\}}}\oplus\Z_{2^{r+1}}\oplus\Z_2, r\geq 1. \label{pi7(Mr4)}
\end{align}

\section{Elementary Chang-complexes}

In this section, we calculate the 6,7-dimensional unstable homotopy groups of elementary  Chang-complexes in $\mathbf{A}_3^2$, i.e.  $\pi_{n}(C_{r}^5)$,  $\pi_{n}(C^{5,s})$,  $\pi_{n}(C_{r}^{5,s})$ for $n=6,7$. Note that  $\pi_{6}(C_{\eta}^5)=\Z_6$ and  $\pi_{7}(C_{\eta}^5)=\Z$ are given by Proposition 8.2 of \cite{Mukai:pi(CPn)}.

In the first Section we denote $j^{n+1}_1$, resp. $j^n_2$ as the canonical inclusion of $S^{n+1}$, resp. $S^n$, into $S^{n+1}\vee S^n$. In the following, for the  special case $n=3$, we simplify the notion $j_1=j^{4}_1$ and $j_2=j^{3}_2$.

\subsection{Fibration sequence and cofibration sequence}
\label{subsec. fibre cofibre}

In order to calculate homotopy groups in a unified and efficient way, we denote the space $\Cse$ which is a mapping cone of $f_{r,\epsilon}^s$, i.e., there is a  cofibration sequence
\begin{align}
	S^4\vee S^3\xrightarrow{f_{r,\epsilon}^s} S^4\vee S^3\xrightarrow{\lambda_{r,\epsilon}^s} \Cse \xrightarrow{q_{r,\epsilon}^s}  S^5\vee S^4
	\label{cofibre Cr^se}
\end{align}
where  $f_{r,\epsilon}^s=(j_1(2^s\iota_4)+\epsilon j_2\eta_3, j_2(2^r\iota_3))$,  $\epsilon=1$ or $0$; $\infty\geq r,s>0$.

Then $C_{r,1}^{5,\infty}=C_{r}^5\vee S^4$; $C_{\infty,1}^{5,s}=C^{5,s}\vee S^4$; $C_{r,1}^{5,s}=\Cs$;  $C_{\infty,0}^{5,s}=M_{2^s}^{4}\vee S^3\vee S^4$.

Note that $f_{r,\epsilon}^sj_1=j_1(2^s\iota_4)+\epsilon j_2\eta_3$; $f_{r,\epsilon}^sj_2=j_2(2^r\iota_3)$.

Let $\Omega(S^5\vee S^4)\xrightarrow{\partial_{r,\epsilon}^{s}} \F\rightarrow \Cse \xrightarrow{q_{r,\epsilon}^s} S^5\vee S^4$   be the homotpy fibration sequence, where $\F\simeq J( M_{f_{r,\epsilon}^s},  S^4\vee S^3)$. From Lemma \ref{lemma Gray}, $\Sigma \F \simeq S^4\vee S^5
\vee S^8\vee S^8\vee S^9\vee A^s_{r,\epsilon}$ where $A^s_{r,\epsilon}$ is a wedge of spheres with dimension $\geq 10$.
\begin{align}
	Sk_{8}(\F)\simeq J_{2}( M_{f_{r,\epsilon}^s},  S^4\vee S^3)=(S^4\vee S^3)\cup_{\gamma_{r,\epsilon}^s}C((S^4\vee S^3)\wedge (S^3\vee S^2))\nonumber
\end{align}
where $\gamma_{r,\epsilon}^s=[id_{S^4\vee S^3}, f_{r,\epsilon}^s]$.
Let
\begin{align}
	&\gamma^s_{r,\epsilon} |_{S^3\wedge S^3}: S^3\wedge S^3\xrightarrow{j_2\wedge j^3_1}(S^4\vee S^3)\wedge (S^3\vee S^2)\xrightarrow{\gamma^s_{r,\epsilon}}S^4\vee S^3;\nonumber\\
	\gamma^s_{r,\epsilon} |_{S^3\wedge S^3}&=\gamma^s_{r,\epsilon} (j_2\wedge j^3_1)=[id_{S^4\vee S^3}, f^s_{r,\epsilon}]( j_2\wedge j^3_1)=[id_{S^4\vee S^3}j_2, f^s_{r,\epsilon} j_1]\nonumber\\
	&=[j_2, j_1(2^s\iota_4)+\epsilon j_2\eta_3]=[j_2, j_1(2^s\iota_4)]+\epsilon[j_2, j_2\eta_3]=2^s[j_1,j_2].\nonumber\\
	\text{Similary,} &\nonumber\\
	\gamma^s_{r,\epsilon}  |_{S^4\wedge S^2}&=\gamma^s_{r,\epsilon}  (j_1\wedge j^2_2)=[id_{S^4\vee S^3}, f^s_{r,\epsilon}]( j_1\wedge j^2_2);\nonumber\\
	&=[id_{S^4\vee S^3} j_1, f^s_{r,\epsilon}j_2]=[j_1, j_2(2^r\iota_3)]=2^r[j_1,j_2];\nonumber\\
	\gamma^s_{r,\epsilon}  |_{S^3\wedge S^2}&=\gamma^s_{r,\epsilon}  (j_2\wedge j^2_2)=[id_{S^4\vee S^3} j_2, f^s_{r,\epsilon} j_2]=[j_2, j_2(2^r\iota_3)]=2^r[j_2,j_2]=0\nonumber\\
	\gamma^s_{r,\epsilon} |_{S^4\wedge S^3}&=\gamma^s_{r,\epsilon} (j_1\wedge j^3_1)=[id_{S^4\vee S^3}, f^s_{r,\epsilon}]( j_1\wedge j^3_1)=[id_{S^4\vee S^3}j_1, f^s_{r,\epsilon} j_1]\nonumber\\
	&=[j_1, j_1(2^s\iota_4)+\epsilon j_2\eta_3]=[j_1, j_1(2^s\iota_4)]+[j_1, \epsilon j_2\eta_3]\nonumber\\
	&=2^{s+1}j_1\nu_4-2^sj_1\Sigma \nu'+\epsilon[j_1,j_2]\eta_6.\nonumber
\end{align}
where $[j_1, j_1(2^s\iota_4)]=2^s[j_1,j_1]=2^sj_1[\iota_4,\iota_4]=2^sj_1(2\nu_4-\Sigma \nu')=2^{s+1}j_1\nu_4-2^sj_1\Sigma \nu'$; $[j_1, j_2\eta_3]=[j_1\Sigma\iota_3, j_2\Sigma\eta_2]=[j_1,j_2]\Sigma \iota_3\wedge \eta_2=[j_1,j_2]\eta_6$; $[j_2,j_1]=(-1)^{(3+1)(4+1)}[j_1,j_2]$ and $[j_2,j_2]=0$ by the injection $ \pi_{5}(S^4\vee S^3)\xrightarrow{\Sigma} \pi_{6}(S^5\vee S^4)$. Hence there is a cofibration sequence
\begin{align}
	S^5\vee S^6\vee S^6\vee S^7 \xrightarrow{\gamma_{r,\epsilon}^s} S^4\vee S^3   \xrightarrow{j^s_{r,\epsilon}} Sk_{8}\F   \xrightarrow{p_{r,\epsilon}^s} S^6\vee S^7\vee S^7\vee S^8 \label{Cof of Sk8Frse}
\end{align}
\begin{align}
	&Sk_{8}(\F)\simeq (S^4\vee S^3)\cup_{\gamma_{r,\epsilon}^s=(0,2^s[j_1,j_2], 2^r[j_1,j_2],\gamma^s_{r,\epsilon} |_{S^4\wedge S^3})}C(S^5\vee S^6\vee S^6\vee S^7)\nonumber\\
	\simeq & (S^4\vee S^3)\cup_{(\gamma^s_{r,\epsilon} |_{S^4\wedge S^3},2^s[j_1,j_2], 2^r[j_1,j_2])}C(S^6\vee S^6\vee S^7)\bigvee S^6 \nonumber.
\end{align}
Let $j_{S^6}:S^6\rightarrow Sk_{8}(\F)$ be the canonical inclusion  of the wedge summand $S^6$ of  $Sk_{8}(\F)$. Simplify the notation $j^{s}_{r,\epsilon}:=j_{q_{r,\epsilon}^s}$: $S^4\vee S^3\hookrightarrow Sk_{8}(\F)$ or  $S^4\vee S^3\hookrightarrow \F$.

\subsection{Calculating $\pi_{6}(C_{r}^{5}) $, $\pi_{6}(C^{5,s})$ and $\pi_{6}(\Cs)$}
\label{subsec:pi6 C}
In the following, $r$ and $s$ cannot be equal to $\infty$  at the same time, unless otherwise stated..

From Lemma \ref{partial calculate1}, we get the exact sequence with two commutative squares
\begin{align} \footnotesize{\xymatrix{
			\!\!\!\!\pi_{7}( S^5\vee S^4)\ar[r]^-{ (\partial_{r,\epsilon}^{s})_{6\ast}}&\pi_{6}(\F)\ar[r]&\pi_{6}(\Cse)\ar[r]^-{q_{r,\epsilon\ast}^s}&\pi_{6}( S^5\!\vee\! S^4)\ar[r]^-{ (\partial_{r,\epsilon}^{s})_{5\ast}}& \pi_{5}(\F)\\
			\!\!\!\!\pi_{6}( S^4\!\vee\! S^3)\ar[r]^-{f^s_{r,\epsilon\ast}}\ar[u]_{\Sigma } &\pi_{6}(S^4\!\vee\! S^3)\ar[u]_{j^{s}_{r,\epsilon\ast} }&&\pi_{5}( S^4\!\vee\! S^3)\ar[r]^-{f^s_{r,\epsilon\ast}}\ar[u]_{\cong\Sigma } &\pi_{5}(S^4\!\vee\! S^3)\ar[u]_{j^{s}_{r,\epsilon\ast} }^{\cong}} } \label{exact seq for pi6(Crse)}
\end{align}
\begin{align}
	&\pi_{5}( S^4\vee S^3)=\Z_2\{j_1\eta_4\}\oplus\Z_2\{j_2\eta_3\eta_4\};\nonumber\\
	&\pi_{6}( S^5\vee S^4)=\Z_2\{\Sigma j_1\eta_5\}\oplus\Z_2\{\Sigma j_2\eta_4\eta_5\}; \nonumber\\
	&\pi_{6}( S^4\vee S^3)=\Z_2\{j_1\eta_4\eta_5\}\!\oplus \!\Z_4\{j_2\nu'\}  \oplus \Z_{(2)}\{[j_1,j_2]\};\nonumber\\
	&\pi_{7}( S^5\vee S^4)=\Z_2\{  j^5_1\eta_5\eta_6\}\oplus \Z_4\{ j^4_2\Sigma \nu'\}  \oplus \Z_{(2)}\{ j^4_2\nu_4\}.\nonumber
\end{align}

By the right commutative square in (\ref{exact seq for pi6(Crse)})\\
$(\partial_{r,\epsilon}^{s})_{5\ast}(j^5_1\eta_5)=j^{s}_{r,\epsilon} f^s_{r,\epsilon} (j_1\eta_4)=j^{s}_{r,\epsilon} (j_1(2^s\iota_4)+\epsilon j_2\eta_3)\eta_4=\epsilon j^{s}_{r,\epsilon} j_2\eta_3\eta_4$.\\
$(\partial_{r,\epsilon}^{s})_{5\ast}(j^4_2\eta_4\eta_5)=j^{s}_{r,\epsilon} f^s_{r,\epsilon}(j_2\eta_3\eta_4)=j^{s}_{r,\epsilon} j_2(2^r\iota_3)\eta_3\eta_4=0$.

Thus $Ker (\partial_{r,\epsilon}^{s})_{5\ast} =\Z_2\{j^4_2\eta_4\eta_5\}$ for $\epsilon=1$.

\begin{lemma}\label{lem pi6Frse}
	$ \pi_6(\F)=\Z_2\{j_{r,\epsilon}^{s}j_1\eta_4\eta_5\}\oplus \Z_4\{j_{r,\epsilon}^{s}j_2\nu'\}  \oplus \Z_{2^{min\{r,s\}}}\{j_{r,\epsilon}^{s}[j_1,j_2]\} \oplus \Z_{(2)}\{j_{S^6}\iota_6\}$
\end{lemma}
\begin{proof}
	From Lemma \ref{stable exact seq} and  the section $j_{S^6}:S^6\rightarrow Sk_{8}(\F)$, the cofibration sequence (\ref{Cof of Sk8Frse}) induces the following exact sequence
	$$ \xymatrix{
		\pi_6(S^5\vee S^6\vee S^6\vee S^7) \xrightarrow{\gamma_{r,\epsilon\ast}^s} \pi_{6}(S^4\vee S^3)\ar[r]^-{j_{r,\epsilon}^{s} } & \pi_6(\F)\ar@{->>}[r]& \pi_{6}(S^6)=\Z_{(2)}\{\iota_6\}\ar@/^1pc/[l]^{j_{S^6\ast}} }$$
	where $\pi_{6}(S^5\vee S^6\vee S^6\vee S^7)=\Z_2\{j'_1\eta_5\}\oplus \Z_{(2)}\{j'_2\iota_6\}\oplus \Z_{(2)}\{j'_3\iota_6\} $, $j'_k$ is the canonical inclusion of the $k$-th wedge summand of $S^5\vee S^6\vee S^6\vee S^7$;
	It is easy to get
	$\gamma_{r,\epsilon\ast}^s(j'_3\eta_5)=0$;$\gamma_{r,\epsilon\ast}^s(j'_2\iota_6)=2^s[j_1,j_2]$; $\gamma_{r,\epsilon\ast}^s(j'_3\iota_6)=2^r[j_1,j_2]$.\\
	Hence one gets $\pi_6(\F)$ by calculating $Coker \gamma_{r,\epsilon\ast}^s$.
\end{proof}
\begin{lemma}\label{lem Cokerpartial6rse}
	\begin{align}
		&Coker (\partial^s_{r,1})_{6\ast}\cong
		\Z_2\oplus (1-\epsilon_r)\Z_2\oplus\Z_{2^{min\{r,s\}}}\oplus\Z_{2^{r+\epsilon_r}}, \infty\geq r\geq 1.\nonumber
	\end{align}
\end{lemma}
\begin{proof}
	By the right commutative square in (\ref{exact seq for pi6(Crse)})
	\begin{align}
		&(\partial_{r,\epsilon}^{s})_{6\ast}( j^5_1\eta_5\eta_6)=(\partial_{r,\epsilon}^{s})_{6\ast}\Sigma  (j_1\eta_4\eta_5)=j_{r,\epsilon}^{s}f^s_{r,\epsilon\ast}(j_1\eta_4\eta_5)=j_{r,\epsilon}^s(f^s_{r,\epsilon}j_1)\eta_4\eta_5\nonumber \\
		&=j_{r,\epsilon}^s(j_1(2^s\iota_4)+\epsilon j_2\eta_3)\eta_4\eta_5=2\epsilon j_{r,\epsilon}^sj_2\nu' \label{partial6 Z2}\\
		&(\partial_{r,\epsilon}^{s})_{6\ast}( j^4_2\Sigma\nu')=(\partial_{r,\epsilon}^{s})_{6\ast}\Sigma (j_2\nu')=j_{r,\epsilon\ast}^sf^s_{r,\epsilon\ast}(j_2\nu')=j_{r,\epsilon}^s(f_{r,\epsilon}^sj_2)\nu'\nonumber \\
		&=j_{r,\epsilon}^s(j_2(2^r\iota_3))\nu'=2^rj_{r,\epsilon}^sj_2\nu'.\label{partial6 Z12}
	\end{align}
	There is a map $\m{r}{3}\xrightarrow{\bar{\theta}}\Cse$ making the following left ladder homotopy commutative and it induces the following right homotopy commutative ladder\\
	\begin{align} \footnotesize{\xymatrix{
				S^3\ar[r]^-{2^r\iota_3}\ar[d]_{j_2} & S^3\ar[r]^{i_3}\ar[d]_{j_2}&  \m{r}{3}\ar[r]^-{p_3}\ar[d]^{\bar{\theta} }& S^4\ar[d]_{j^4_2}\\
				S^4\vee S^3 \ar[r]^{f^s_{r,\epsilon}} &  S^4\vee S^3\ar[r]&\Cse \ar[r]^{q^s_{r,\epsilon}}&S^5\vee S^4,}}
		\footnotesize{\xymatrix{
				\Omega S^4\ar[r]^-{\partial}\ar[d]_{\Omega j^4_2} & F_{p_3}\ar[r]^{i_3}\ar[d]^{\theta}&  \m{r}{3}\ar[d]^{\bar{\theta} }\\
				\Omega (S^5\vee S^4) \ar[r]^{\partial^s_{r,\epsilon}} &  \F\ar[r]&\Cse, }} \label{diagramMr to Crse}
	\end{align}
	where $\theta j_{p_3}\simeq j^s_{r,\epsilon}j_2$, i.e., $ \footnotesize{\xymatrix{S^3\ar@/_0.5pc/[rr]_{j^s_{r,\epsilon}j_2}\ar@{^{(}->}[r]^{j_{p_3}}   &  F_{p_3} \ar[r]^{\theta} & \F }}$ by Lemma \ref{J(X,A)toJ(X',A')}.
	So we get the following  commutative ladder
	\begin{align}
		\small{\xymatrix{
				\pi_{7}(S^4)\ar[d]_{j^4_{2\ast}} \ar[r]^-{\partial_{6\ast}} & \!\pi_{6}(F_{p_3})\ar[r]^-{Proj.}\ar[d]_{\theta_\ast }& \pi_{6}( S^3\wedge S^3)\cong \pi_{6}(S^6) \ar[d]_{(j_2\wedge j_2)_\ast=id }\\
				\!\!\!\pi_{7}(S^5\!\vee \!S^4) \ar[r]^-{(\partial^s_{r,\epsilon})_{6\ast}} & \pi_{6}(\F)\ar[r]^-{p_{r,\epsilon\ast}^s}&\pi_{6}((S^4\vee S^3)\wedge (S^4\vee S^3))\cong \pi_{6}(S^6) } }
		\label{ladder pi Mr to Cr^s}
	\end{align}
	\begin{align}
		\!\! \!\!\!\!(\partial^s_{r,\epsilon})_{6\ast}(j^4_{2}\nu_4) \!=\!\theta_{\ast}\partial_{6\ast}(\nu_4)\!=\!\theta_{\ast}(yj_{p_3}\nu'+2^r\!j^6_{p_3})=(y\!+\!2^r\!m)j^s_{r,\epsilon}j_2\nu'\!\pm\!2^r\!j_{S^6}\iota_6
		\label{partial6 j2^4nu4 mv}
	\end{align}
	where $y$ comes from (\ref{equation y}) and $m$  comes from the assumption
	\begin{align}
		\theta_{\ast} (j^6_{p_3})=\widehat{l}j_{r,\epsilon}^{s}j_1\eta_4\eta_5+ mj_{r,\epsilon}^{s}j_2\nu'+\widehat{u}j_{r,\epsilon}^{s} [j_1,j_2]\pm j_{S^6}\iota_6,
		\label{Asumption m,v }
	\end{align}
	for some $\widehat{l}\in \Z_2$, $m\in \Z_4$, $\widehat{u}\in \Z_{2^{min\{r,s\}}}$.
	
	From (\ref{partial6 Z2}),(\ref{partial6 Z12}) and (\ref{partial6 j2^4nu4 mv}) , for $\epsilon=1$, we get
	\begin{align}
		&Coker (\partial^s_{r,1})_{6\ast}\nonumber\\
		=&\frac{\Z_2\{j_{r,\epsilon}^{s}j_1\eta_4\eta_5\}\oplus \Z_4\{j_{r,\epsilon}^{s}j_2\nu'\}  \oplus \Z_{2^{min\{r,s\}}}\{j_{r,\epsilon}^{s}[j_1,j_2]\} \oplus \Z_{(2)}\{j_{S^6}\iota_6\}}{\left \langle  2 j_{r,\epsilon}^{s}j_2\nu', 2^rj_{r,\epsilon}^{s}j_2\nu',  (y+2^rm)j_{r,\epsilon}^{s}j_2\nu'\pm 2^rj_{S^6}\iota_6 \right \rangle}  \nonumber\\
		\cong&\frac{\Z\{a,b,c,d\}}{\left \langle 2a, 2 b, 2^{min\{r,s\}}c, yb\!\pm\!2^rd \right \rangle} \!\cong
		\Z_2\!\oplus\!(1\!-\!\epsilon_r)\Z_2\!\oplus\!\Z_{2^{min\{r,s\}}}\!\oplus\!\Z_{2^{r+\epsilon_r}}, \infty\geq r\geq1 .\nonumber
	\end{align}
\end{proof}

As the proof of the split of (\ref{exact pi6(Mr3)}),  there is also an element  $\bar{\theta}\varsigma_r\in \pi_{6}(C_{r,1}^{5,s})$ with order 2 such that $q_{r,1\ast}^s(\varsigma_r)=j_{2}^4\eta_4\eta_5$. Hence the short exact sequence $\xymatrix{
	Coker(\partial^s_{r,1})_{6\ast}\ar@{^{(}->}[r]  &\pi_{6}(C_{r,1}^{5,s}) \ar@{->>}[r]^-{q_{r,1\ast}^s}& Ker(\partial^s_{r,1})_{5\ast} =\Z_2\{j_{2}^4\eta_4\eta_5\}}$ splits.

So,
$\pi_{6}(C_{r,1}^{5,s})\cong\Z_2\oplus \Z_2\oplus(1-\epsilon_r)\Z_2\oplus\Z_{2^{min\{r,s\}}}\oplus\Z_{2^{r+\epsilon_r}}, \infty\geq r\geq 1$ .

$\pi_6(C_{r,1}^{5,\infty})=\pi_6(C_{r}^5\vee S^4)\cong  \pi_6(C_{r}^5)\oplus \pi_6(S^4)\oplus \pi_6(C_{r}^8)\cong \pi_6(C_{r}^5)\oplus \Z_2 \oplus \Z_{2^{r+\epsilon_r}}$, which implies
$\pi_{6}(C_{r}^{5})\cong
\Z_2\oplus (1-\epsilon_r)\Z_2\oplus\Z_{2^{r+\epsilon_r}}, r\geq 1$.

$\pi_6(C_{\infty,1}^{5,s})=\pi_6(C^{5,s}\vee S^4)\cong  \pi_6(C^{5,s})\oplus \pi_6(S^4)\oplus \pi_6(C^{8,s})\cong \pi_6(C^{5,s})\oplus \Z_2 \oplus \Z_{(2)}$, which implies
$\pi_{6}(C^{5,s})\cong\Z_2\oplus \Z_2\oplus\Z_{2^{s}}$, $s\geq 1$.

\subsection{Calculating $\pi_{7}(C_{r}^{5}) $, $\pi_{7}(C^{5,s})$ and $\pi_{7}(\Cs)$}
\label{subsec:pi7 C}
From Lemma \ref{partial calculate1}, we get the exact sequence with two commutative squares
\begin{align} \small{\xymatrix{
			\pi_{8}( S^5\vee S^4)\ar[r]^-{ (\partial_{r,\epsilon}^{s})_{7\ast}}&\pi_{7}(\F)\ar[r]&\pi_{7}(\Cse)\ar[r]^-{q_{r,\epsilon\ast}^s}&\pi_{7}( S^5\vee S^4)\ar[r]^-{ (\partial_{r,\epsilon}^{s})_{6\ast}}& \pi_{6}(\F)\\
			\pi_{7}( S^4\vee S^3)\ar[r]^-{f^s_{r,\epsilon\ast}}\ar[u]_{\Sigma } &\pi_{7}(S^4\vee S^3)\ar[u]_{j^{s}_{r,\epsilon\ast} }&& &} } \label{exact seq for pi7(Crse)}
\end{align}
\begin{align}
	\pi_{7}(S^4\vee S^3)&=\Z_{(2)}\{j_1\nu_4\}\oplus\Z_4\{j_1\Sigma\nu'\}\oplus\Z_2\{j_2\nu'\eta_6\}\oplus\Z_2\{[j_1,j_2]\eta_6\}\nonumber,\\
	\pi_{8}(S^5\vee S^4)&=\Z_8\{j^5_1\nu_5\}\oplus\Z_2\{j_2^4\Sigma\nu'\eta_7\}\oplus\Z_2\{j_2^4\nu_4\eta_7\}\oplus\Z_{(2)}\{[j_1^5,j_2^4]\}\nonumber.
\end{align}
From (\ref{equation y}), (\ref{partial6 Z2}), (\ref{partial6 Z12}), (\ref{partial6 j2^4nu4 mv}), we get\\
\begin{align}
	Ker (\partial_{r,\epsilon}^{s})_{6\ast}=\left\{
	\begin{array}{ll}
		\Z_4\{\epsilon_rj_1^5\eta_5\eta_6+j_2^4\Sigma\nu'\}, & \hbox{$r\geq 1,\epsilon=1$;} \\
		\Z_{(2)}\{j_2^4\nu_4\}\oplus\Z_4\{j_2^4\Sigma\nu'\}, & \hbox{$r=\infty,\epsilon=1$;}\\
		\pi_{7}( S^5\vee S^4),& \hbox{$r=\infty, \epsilon=0$.}\\
	\end{array}
	\right.  \label{Kerpartial rse 6}
\end{align}

In the following we also allow $(r,s,\epsilon)=(0,0,0)$. Then  $C_{0,0}^{5,0}\simeq (S^4\vee S^3)\bigcup_{id}C(S^4\vee S^3)\simeq \ast$ and all the results in Section \ref{subsec. fibre cofibre} hold.

\begin{lemma}\label{lem pi7Frse}
	$ \pi_7(\F)=\Z_2\{j_{S^6}\eta_6\}\oplus \Z_{(2)}\{\widetilde{\rho}_{r,\epsilon}^s\}  \oplus Coker(\gamma_{r,\epsilon }^{s})_{7\ast}$,  where $\widetilde{\rho}_{r,\epsilon}^s\in  \pi_{7}(\F)$ is a lift of $\rho_{r,\epsilon}^s$ in $(\ref{rhorse})$ and $Coker(\gamma_{r,\epsilon }^{s})_{7\ast}$ is given by $(\ref{Cokergammarse})$.
	
\end{lemma}

\begin{proof}
	There is a  fibration sequence for the map $p_{r,\epsilon}^s$ in (\ref{Cof of Sk8Frse})
	\begin{align}
		\Omega(S^6\vee S^7\vee S^7\vee S^8) \xrightarrow{\widehat{\partial_{r,\epsilon}^{s}}}F_{p_{r,\epsilon}^s}  \rightarrow  Sk_{8}\F   \xrightarrow{p_{r,\epsilon}^s} S^6\vee S^7\vee S^7\vee S^8 \nonumber.
	\end{align}
	where $F_{p_{r,\epsilon}^s}\simeq J(M_{\gamma_{r,\epsilon}^s}, S^5\vee S^6\vee S^6\vee S^7)$, with $Sk_{8}F_{p_{r,\epsilon}^s}\simeq S^4\vee S^3\vee S^8$.
	From  Lemma \ref{partial calculate1} and Lemma \ref{stable exact seq},  there is an exact sequence
	\begin{align}
		\pi_{7}(X) \!\xrightarrow{(\gamma_{r,\epsilon }^{s})_{7\ast}} \pi_{7}( S^4\vee S^3 ) \rightarrow  \pi_{7}( \F )  \xrightarrow{p^s_{r,\epsilon\ast}} \pi_{7}( \Sigma X)\cong  \pi_{6}( X) \xrightarrow{(\gamma_{r,\epsilon }^{s})_{6\ast}} \pi_{6}(S^4\vee S^3) \nonumber.
	\end{align}
	where  $X=S^5\vee S^6\vee S^6\vee S^7$ and $\pi_6(X)=\Z_2\{j'_1\eta_5\}\oplus \Z_{(2)}\{j'_2\iota_6\}\oplus \Z_{(2)}\{j'_3\iota_6\}$. Let $\iota^{34}_7 :=\Sigma j'_2\iota_7 $ and $\iota^{43}_7:=\Sigma j'_3\iota_7$.
	$$ Ker(\gamma_{r,\epsilon }^{s})_{6\ast}\cong \Z_{(2)}\{\rho_{r,\epsilon}^s\}\oplus\Z_2\{\Sigma j'_1 \eta_6\},$$
	\begin{align}
		\!\!\!\!\!\!\footnotesize{\begin{tabular}{r|c|c|c|}
				\cline{2-4}
				& $\infty\geq s>  r>0$&$\infty\geq r>s>0$&$r=s$ \\
				\cline{2-4}
				$\rho_{r,\epsilon}^s=$  &  $ 2^{s-r}\iota^{43}_7-\iota^{34}_7$ & $ \iota^{43}_7-2^{r-s}\iota^{34}_7$ &$\iota^{43}_7-\iota^{34}_7$\\
				\cline{2-4}
		\end{tabular}}~  \label{rhorse}
	\end{align}
	\begin{align}
		&Coker(\gamma_{r,\epsilon }^{s})_{7\ast}=\frac{\Z_{(2)}\{j_1\nu_4\}\oplus\Z_4\{j_1\Sigma\nu'\}\oplus\Z_2\{j_2\nu'\eta_6\}\oplus\Z_2\{[j_1,j_2]\eta_6\}}{\left \langle 2^{s+1}j_1\nu_4-2^sj_1\Sigma \nu'+\epsilon[j_1,j_2]\eta_6\right\rangle} .\label{Cokergammarse}
	\end{align}
	
	$0\rightarrow Coker(\gamma_{r,\epsilon }^{s})_{7\ast}\xrightarrow{j_{r,\epsilon,\ast}^s} \pi_7(\F)\xrightarrow{p^s_{r,\epsilon\ast}} \Z_{(2)}\{\rho_{r,\epsilon}^s\}\oplus\Z_2\{\Sigma j'_1 \eta_6\}\rightarrow 0$.
	
	Above exact sequence splits since in (\ref{Cof of Sk8Frse}), the wedge summand $S^6$ of $ S^6\vee S^7\vee S^7\vee S^8$ has a section $j_{S^6}:S^6\rightarrow Sk_{8}(\F)$. Thus we complete the proof of this lemma.
\end{proof}

From the commutative square in (\ref{exact seq for pi7(Crse)})
\begin{align}
	&(\partial_{r,\epsilon}^{s})_{7\ast}( j^5_1\nu_5)=(\partial_{r,\epsilon}^{s})_{7\ast}\Sigma  (j_1\nu_4)=j_{r,\epsilon,\ast}^sf^s_{r,\epsilon\ast}(j_1\nu_4)=j_{r,\epsilon}^s(f^s_{r,\epsilon}j_1)\nu_4\nonumber \\
	&=j_{r,\epsilon}^s(j_1(2^s\iota_4)+\epsilon j_2\eta_3)\nu_4=j_{r,\epsilon}^s(j_1(2^s\iota_4)\nu_4+\epsilon j_2\eta_3\nu_4+[j_1(2^s\iota_4),\epsilon j_2\eta_3]H(\nu_4)) \nonumber\\
	&=2^{2s}j_{r,\epsilon}^sj_1\nu_4-2^{s-1}(2^s-1)j_{r,\epsilon}^sj_1\Sigma\nu'+ \epsilon j_{r,\epsilon}^sj_2\nu'\eta_6\label{partial7 j1nu5}\\
	&(\partial_{r,\epsilon}^{s})_{7\ast}( j^4_2\Sigma\nu'\eta_7)=(\partial_{r,\epsilon}^{s})_{7\ast}\Sigma (j_2\nu'\eta_6)=j_{r,\epsilon}^sf^s_{r,\epsilon}(j_2\nu'\eta_6)=j_{r,\epsilon}^s(f_{r,\epsilon}^sj_2)\nu'\eta_6\nonumber \\
	&=j_{r,\epsilon}^sj_2(2^r\iota_3)\nu'\eta_6=0.~~\text{(Note $\eta_3\nu_4=\nu'\eta_6$.)}\label{partial7j^4Sigmanu'eta_7}
\end{align}
From  (\ref{partial7nu4eta7}) and  the commutative  diagram (\ref{diagramMr to Crse})
\begin{align}
	&(\partial_{r,\epsilon}^{s})_{7\ast}( j^4_2\nu_4\eta_7)=\theta_{\ast}\partial_{7\ast}(\nu_4\eta_7)= \theta_{\ast}(\epsilon_{r}j_{p_3}\nu'\eta_6)=\epsilon_{r}\theta j_{p_3}\nu'\eta_6=\epsilon_{r} j_{r,\epsilon}^sj_2\nu'\eta_6. \label{partial7 j2^4nu4eta7}
\end{align}
It remains to compute
$(\partial^s_{r,1})_{7,*}([j^5_1,j^4_2])$  the determination of which requires the computation of $(\partial^s_{\infty,0})_{7,*}([j^5_1,j^4_2])$.

Since $Coker(\gamma_{r,1}^{s})_{7\ast}=\Z_{2^{s+2}}\{j_{r,\epsilon}^{s}j_1\nu_4\}\!\oplus\! \Z_4\{j_{r,\epsilon}^{s}j_1\Sigma\nu'\}\!\oplus\! \Z_2\{j_{r,\epsilon}^{s}j_2\nu'\eta_6\}$ in (\ref{Cokergammarse}), suppose
\begin{align}
	\!\! (\partial_{r,1}^{s})_{7\ast}([j^5_1,j^4_2])=x\widetilde{\rho}_{r,1}^s+yj_{S^6}\eta_6+kj_{r,1}^sj_1\nu_4
	+lj_{r,1}^sj_1\Sigma\nu'+uj_{r,1}^sj_2\nu'\eta_6 \label{partial7r1s [j15j24]}
\end{align}
where $y,u\in \Z_2$; $l\in \Z_4$; $k\in \Z_{2^{s+2}}$; $x\in \Z$ are to be determined.

By simplifying  the   $j^s_{\infty,0}:S^4\vee S^3\hookrightarrow Sk_{8}(F_{\infty,0}^s)$ by $j_{0}^s$,
we also suppose
\begin{align}
	(\partial_{\infty,0}^s)_{7\ast}([j_1^5,j_2^4])\!=\! t'\widetilde{\iota_{7}}^{43}\!+\!y'j_{S^6}\eta_6\!+\!vj^s_{0}j_1\nu_4\!+\!wj^s_{0}j_1\Sigma\nu'\!+\!u'j^s_{0}j_2\nu'\eta_6\!+\!zj^s_{0}[j_1,j_2]\eta_6, \label{partial7infty,0s [j15j24]}
\end{align}
where $t',v\in \Z$, $y',u',z\in \Z_2$, $w\in \Z_4$.

The determination of the first two coefficients in (\ref{partial7r1s [j15j24]}) and (\ref{partial7r1s [j15j24]}) can be done simultaneously in the following Lemma.
\begin{lemma}\label{lem detemine x,y,t'}
	In $(\ref{partial7r1s [j15j24]})$, $(\ref{partial7infty,0s [j15j24]})$, $y=1, y'=0\in \Z_2$; $x=2^{min\{r,s\}}t$;  $t'=2^st$, $t$ is odd.
\end{lemma}
\begin{proof}
	
	There is a map $C_{0,0}^{5,0}\xrightarrow{\bar{\theta_0}}C_{r,\epsilon}^{5,s}$ making  the following left ladder homotopy commutative and it induces the following right homotopy commutative ladder\\
	\begin{align} \footnotesize{\xymatrix{
				S^4\vee S^3\ar[r]^-{id}\ar@{=}[d]_{id} &S^4\vee  S^3\ar[r]\ar[d]_{f_{r,\epsilon}^s}& C_{0,0}^{5,0} \ar[r]^-{q^0_{0,0}}\ar[d]^{\bar{\theta_0} }& S^5\vee S^4\ar@{=}[d]_{id}\\
				S^4\vee S^3 \ar[r]^{f^s_{r,\epsilon}} &  S^4\vee S^3\ar[r]&C_{r,\epsilon}^{5,s} \ar[r]^{q^s_{r,\epsilon}}&S^5\vee S^4,}}
		\footnotesize{\xymatrix{
				\Omega (S^5\vee S^4)\ar[r]^-{\partial^0_{0,0}}\ar@{=}[d]_{id} & F^0_{0,0}\ar[r]\ar[d]^{\theta_0}&  C_{0,0}^{5,0}\ar[d]^{\bar{\theta_0} }\\
				\Omega (S^5\vee S^4) \ar[r]^{\partial^s_{r,\epsilon}} &  \F\ar[r]&C_{r,\epsilon}^{5,s}. }} \nonumber
	\end{align}
	We get the following  two commutative diagrams
	\begin{align}
		\small{\xymatrix{
				\pi_{8}(S^5\vee S^4)\ar@{=}[d]_{id} \ar[r]^-{(\partial^0_{0,0})_{7\ast}} & \!\pi_{7}(F^0_{0,0})\ar[d]_{\theta_{0\ast} }\\
				\!\!\!\pi_{8}(S^5\!\vee \!S^4) \ar[r]^-{(\partial^s_{r,\epsilon})_{7\ast}} & \pi_{7}(F_{r,\epsilon}^s),} }~~\footnotesize{\xymatrix{
				\pi_7 (S^4\vee S^3)\ar[r]^-{j_{0,0\ast}^{0}}\ar@{=}[d]_{id} &\pi_7 ( F_{0,0}^{0})\ar[r]^-{p_{0,0\ast}^{0}}\ar[d]^{\theta_{0\ast}}& \pi_7(S^6\vee S^7\vee S^7)\ar[d]^{\widetilde{\theta_0}_{\ast}}\\
				\pi_7 (S^4\vee S^3) \ar[r]^-{j^s_{r,\epsilon \ast}} & \pi_7(\F)\ar[r]^-{p^s_{r,\epsilon \ast}}&\pi_7(S^6\vee S^7\vee S^7), }} \nonumber
	\end{align}
	where $\widetilde{\theta_0}=(f_{r,\epsilon}^s\wedge id)|_{S^6\vee S^7\vee S^7}$ by Lemma \ref{J(X,A)toJ(X',A')}.
	and  $\Omega(S^5\vee S^4)  \xrightarrow{\partial^0_{0,0} } F^0_{0,0}$ is a homotopy equivalence.
	Thus from $Coker(\gamma_{0,0}^{0})_{7\ast}=\Z_8\{j_1\nu_4\}\oplus \Z_2\{j_2\nu'\eta_6\}\oplus\Z_2\{[j_1,j_2]\eta_6\}$ in (\ref{Cokergammarse}) we get\\ $(\partial^0_{0,0})_{7\ast}([j_1^5,j_2^4])=t\widetilde{\rho^0_{0,0}}+y_0j_{S^6} \eta_6+k_0j^0_{0,0}j_1\nu_4+ u_0j^0_{0,0}j_2\nu'\eta_6+w_0j^0_{0,0}[j_1,j_2]\eta_6$\\ where $t$ is odd integer, $y_0, u_0, w_0\in \Z_2$,$k_0\in \Z_8$.
	\begin{align}
		&p_{r,\epsilon\ast}^s(\partial_{r,\epsilon}^s)_{7\ast}([j_1^5,j_2^4])=p_{r,\epsilon\ast}^s\theta_{0\ast}(\partial^0_{0,0})_{7\ast}([j_1^5,j_2^4])\nonumber \\
		=&\widetilde{\theta_0}_{\ast}p^0_{0,0\ast}(t\widetilde{\rho}^0_{0,0}+y_0j_{S^6} \eta_6+k_0j^0_{0,0}j_1\nu_4+ u_0j^0_{0,0}j_2\nu'\eta_6+w_0j^0_{0,0}[j_1,j_2]\eta_6)\nonumber \\
		=&\widetilde{\theta_0}_{\ast}(t(\iota^{43}_7-\iota^{34}_7)+y_0\Sigma j'_1 \eta_6 )=t(2^s\iota^{43}_7-2^r\iota^{34}_7)+t\epsilon\Sigma j'_1 \eta_6.\nonumber\\
		=& 2^{min\{r,s\}}t\rho_{r,\epsilon}^s+t\epsilon\Sigma j'_1 \eta_6\nonumber
	\end{align}
	On the other hand, from (\ref{partial7r1s [j15j24]}) and (\ref{partial7infty,0s [j15j24]})
	\begin{align}
		&p_{r,1\ast}^s(\partial_{r,1}^s)_{7\ast}([j_1^5,j_2^4])=p_{r,1\ast}^s(x\widetilde{\rho}_{r,1}^s+yj_{S^6}\eta_6+kj_{r,1}^sj_1\nu_4
		+\dots)=x\rho_{r,1}^s+y\Sigma j'_1 \eta_6\nonumber\\
		&p_{\infty,0\ast}^s(\partial_{\infty,0}^s)_{7\ast}([j_1^5,j_2^4])=p_{\infty,0\ast}^s(  t'\widetilde{\iota_{7}}^{43}+y'j_{S^6}\eta_6+vj^s_{0}j_1\nu_4+\dots)=t'\iota^{43}_7+ y'\Sigma j'_1 \eta_6\nonumber
	\end{align}
	So $y= t=1 \in \Z_2$, $y'=0\in \Z_2$; $x=2^{min\{r,s\}}t$; $t'=2^{min\{\infty,s\}}t=2^st$, $t$ is odd.
\end{proof}
The determination of the remaining coefficients in (\ref{partial7r1s [j15j24]}) depends on the remaining coefficients in (\ref{partial7infty,0s [j15j24]}).

The following short exact sequence is split since $\pi_{7}(S^5)$  splits out of  $\pi_{7}(\m{s}{4})$.
\begin{align}
	Coker (\partial_{\infty,0}^s)_{7\ast}   \hookrightarrow\!  \pi_{7}(\m{s}{4}\!\vee\! S^3\!\vee\! S^4)\!\twoheadrightarrow\! \pi_{7}(S^5\! \vee\! S^4)\cong  \pi_{7}(S^5)\!\oplus \! \pi_{7}(S^4)\nonumber
\end{align}
Note that $(\partial_{\infty,0}^s)_{7\ast}$ is given by (\ref{partial7 j1nu5})  (\ref{partial7j^4Sigmanu'eta_7})  (\ref{partial7 j2^4nu4eta7})  (\ref{partial7infty,0s [j15j24]}),i.e.,

$ Coker (\partial_{\infty,0}^s)_{7\ast}=\Z_2\{j_{S^6}\eta_6\}\oplus H^s$, where $H^s$ is given by the following
\begin{align}
	\begin{array}{l}
		\frac { \Z_{(2)}\{\widetilde{\iota_{7}}^{43}\}\oplus \Z_{2^{s+1}}\{j^s_{0}j_1\nu_4\}\oplus\Z_4\{j^s_{0}j_1\Sigma\nu'\}\oplus \Z_2\{j^s_{0}j_2\nu'\eta_6\}\oplus\Z_2\{j^s_{0}[j_1,\!j_2]\eta_6\}}
		{ \left \langle 2^{s+1}j^s_{0}\!j_{1}\!\nu_4\!-\!2^{s}j^s_{0}\!j_{1}\!\Sigma\nu'\!,~  2^{2s}j^s_{0}\!j_{1}\!\nu_4\!-\!2^{s\!-\!1}(2^s\!-\!1)j^s_{0}\!j_{1}\!\Sigma\nu'\!, ~2^st\widetilde{\iota_{7}}^{43}\!+\!vj^s_{0}\!j_1\!\nu_4\!+\!wj^s_{0}\!j_1\!\Sigma\nu'
			+u'j^s_{0}\!j_2\!\nu'\eta_6\!+\!zj^s_{0}\![j_1,j_2]\eta_6 \right \rangle } \\
	\end{array}.  \label{Coker (partial0s)7ast vw}
\end{align}
On the other hand, by (\ref{pi7(Mr4)}),  we have

$ \pi_{7}(\m{s}{4}\vee S^3\vee S^4)\cong \pi_{7}(\m{s}{4})\oplus \pi_{7}(S^3)\oplus \pi_{7}(S^4) \oplus \pi_{7}(\m{s}{4}\wedge S^2)\oplus \pi_{7}(\m{s}{4}\wedge S^3)\oplus \pi_{7}(S^{3}\wedge S^3)\cong \Z_{2^{s+1}}\oplus \Z_{2^{\alpha_s}}\oplus \Z_2\oplus \Z_2\oplus \Z_{(2)}\oplus \Z_4\oplus \Z_2\oplus\Z_{2^s}\oplus \Z_2$, where $\alpha_{s}=min\{2,s-1\}$.
\begin{align}
	\text{Thus}&~~~~~H^s\cong \Z_{2^{s+1}}\oplus \Z_{2^{\alpha_s}}\oplus\Z_{2^s}\oplus \Z_2\oplus \Z_2. \label{H^s}
\end{align}
\begin{lemma}\label{lem detemine v,w}
	In $(\ref{partial7infty,0s [j15j24]})$, for $s\geq 1$ , $2^{\alpha_s}\mid w$, $2^{s}\mid v$.
\end{lemma}
The proof of the Lemma is elementary and will be postponed to the Appendix. Assuming the Lemma one gets the remaining coefficients in (\ref{partial7r1s [j15j24]}) as follows.

\begin{lemma}\label{lem k l}
	$k=2^{min\{r,s\}}k'$ and  $l=2^{min\{s-1,1\}}l'$,  for some $k',l'\in \Z$.
\end{lemma}
\begin{proof}
	We have following commutative diagrams  by Lemma \ref{J(X,A)toJ(X',A')}, where $F_{r,1}^{s}\xrightarrow{\chi}F_{\infty,0}^{s}$ is induced by the map $C_{r,1}^{5,s}\xrightarrow{\bar{\chi}}  C_{\infty,0}^{5,s}$ in the right commutative diagrams
	\begin{align}
		\footnotesize{\xymatrix{
				\pi_8(S^5\vee S^4)\ar[r]^-{(\partial_{r,1}^{s})_{7\ast}}\ar@{=}[d]_{id} &\pi_7 ( F_{r,1}^{s})\ar[r]^-{p_{r,1\ast }^{s}}\ar[d]^{\chi_{\ast}}& \pi_7((S^4\vee S^3)^{\wedge 2})\ar[d]^{(\bar{j}_1\wedge id)_{\ast}}\\
				\pi_8 (S^5\vee S^4) \ar[r]^-{(\partial_{\infty,0}^s)_{7\ast}} & \pi_7(F_{\infty,0}^s)\ar[r]^-{p_{\infty,0\ast}^s}&\pi_7((S^4\vee S^3)^{\wedge 2}), }}
		\footnotesize{\xymatrix{
				S^4\vee S^3\ar[r]^-{f_{r,1}^{s}}\ar@{=}[d]_{id} &S^4\vee  S^3\ar[r]\ar[d]_{\bar{j}_1=(j_1, 0)}& C_{r,1}^{5,s} \ar[d]^{\bar{\chi} }\\
				S^4\vee S^3 \ar[r]^{f^s_{\infty,0}} &  S^4\vee S^3\ar[r]& C_{\infty,0}^{5,s},}}  \label{diag Crs1 to C infty s0}
	\end{align}
	where $\chi j_{r,1}^{s}=j_{0}^s\bar{j}_1$ , $\footnotesize{\xymatrix{S^4\vee S^3\ar@/_1pc/[rr]_{j_{0}^s\bar{j}_1}\ar@{^{(}->}[r]^{j_{r,1}^s}   &  F_{r,1}^{s}\ar[r]^{\chi} &F_{\infty,0}^s }}$.
	
	By noting that $j_{S^6}\eta_6$ is 2-torsion, for $s\geq 1$, we suppose that
	\begin{align}
		&\chi_\ast(j_{S^6}\eta_6)\!=\!2^st'_1j_{0}^sj_1\nu_4+2^{min\{s-1,1\}}t'_2j_{0}^sj_1\Sigma\nu'
		+t'_3j_{0}^sj_2\nu'\eta_6+t'_4j_{0}^s[j_1,j_2]\eta_6, \nonumber\\
		&\chi_\ast(\widetilde{\rho}_{r,1}^s)\!=
		x_{r}^s\widetilde{\iota_{7}}^{43}+t_1j_{0}^sj_1\nu_4+t_2j_{0}^sj_1\Sigma\nu'
		+t_3j_{0}^sj_2\nu'\eta_6+t_4j_{0}^s[j_1,j_2]\eta_6.\nonumber
	\end{align}
	where $x_r^s=\left\{\!\!\!
	\begin{array}{ll}
		1, & \hbox{$s\leq r$;} \\
		2^{s-r}, & \hbox{$s\geq r$.}
	\end{array}
	\right.$
	$t_1,t'_1\in \Z, t_2,t_2'\in \Z_4, t_3, t_3',t_4, t_4'\in \Z_2$.

	From (\ref{partial7r1s [j15j24]}) and Lemma \ref{lem detemine x,y,t'}
	\begin{align}
		&\chi_{\ast}(\partial_{r,1}^s)_{7\ast}([j_1^5,j_2^4])=\chi_{\ast}(2^{min\{r,s\}}t\widetilde{\rho}_{r,1}^s+j_{S^6}\eta_6+kj_{r,1}^sj_1\nu_4
		+lj_{r,1}^sj_1\Sigma\nu'+uj_{r,1}^sj_2\nu'\eta_6) \nonumber\\
		&=2^st \widetilde{\iota_{7}}^{43}+(2^{min\{r,s\}}tt_1+2^st_1'+k)j_{0}^sj_1\nu_4+(2^{min\{r,s\}}tt_2+2^{min\{s-1,1\}}t_2'+l)j_{0}^sj_1\Sigma\nu'\nonumber.
	\end{align}
	By the left commutative diagram in (\ref{diag Crs1 to C infty s0}), $\chi_{\ast}(\partial_{r,1}^s)_{7\ast}([j_1^5,j_2^4])$ also equals to
	\begin{align}
		&(\partial_{\infty,0}^s)_{7\ast}([j_1^5,j_2^4])
		=2^st\widetilde{\iota_{7}}^{43}+vj^s_{0}j_1\nu_4+wj^s_{0}j_1\Sigma\nu'
		+u'j^s_{0}j_2\nu'\eta_6\!+\!zj^s_{0}[j_1,j_2]\eta_6.\nonumber
	\end{align}
	Thus~ $2^{min\{r,s\}}tt_1+2^st_1'+k=v$ and $2^{min\{r,s\}}tt_2+2^{min\{s-1,1\}}t_2'+l=w$.
	
	By Lemma \ref{lem detemine v,w}, $k=2^{min\{r,s\}}k'$ and  $l=2^{min\{s-1,1\}}l'$,  $k',l'\in \Z$.
	
\end{proof}

\begin{remark}
	The case $s=\infty$ is not allowed in the above proof of Lemma \ref{lem k l}, since we get the maps $F_{r,1}^{\infty}\xrightarrow{\chi}F_{\infty,0}^{\infty}$ and $C_{r,1}^{5,\infty}\xrightarrow{\bar{\chi}}  C_{\infty,0}^{5,\infty}\simeq (S^4\vee S^3)\bigcup_{0}C(S^4\vee S^3)$ where the targets of the maps are not covered by Lemma \ref{lem detemine v,w}. However the fibration  $F_{\infty,0}^{\infty}\rightarrow C_{\infty,0}^{5,\infty} \rightarrow S^5\vee S^4$  splits, which implies that $(\partial_{\infty,0}^\infty)_{7\ast}=0$ in the left commutative diagram of  (\ref{diag Crs1 to C infty s0}). Hence it is easy to see that Lemma \ref{lem k l} is also true for $s=\infty$.
\end{remark}
From Lemma \ref{lem detemine x,y,t'} and Lemma \ref{lem k l}, one gets $ Coker(\partial_{r,1}^s)_{7\ast} $ in the following Lemma  whose proof is also postponed to the Appendix.
\begin{lemma}\label{lem Cokerpartial rs1}
	$Coker(\partial_{r,1}^{s})_{7\ast}\cong\Z_{2^{min\{s-\epsilon_r,2\}}}\oplus\Z_{2^{min\{s+1,r+1\}}}\oplus\Z_{2^{s+2}}, \infty\geq r\geq 1.$
\end{lemma}

\begin{lemma}\label{split for pi7(Crs)}
	The following short exact sequence is split for $\infty\geq r\geq 1$.
	\begin{align}
		&0\!\rightarrow\! Coker(\partial_{r,1}^s)_{7\ast} \!\rightarrow\! \pi_{7}(C_{r,1}^{5,s})\rightarrow Ker(\partial_{r,1}^s)_{6\ast}\rightarrow \!0 \label{exact for pi7(Crs1)}
	\end{align}
	where $Ker(\partial_{r,1}^s)_{6\ast}$ and $Coker(\partial_{r,1}^s)_{7\ast}$ are given by (\ref{Kerpartial rse 6}) and Lemma \ref{lem Cokerpartial rs1}.
\end{lemma}

\begin{proof}
	For $r\geq 2$, there is $\alpha\in \pi_{7}(\m{r}{3})$ with order 4, which is a lift of $\Sigma\nu'\in Ker\partial_{6\ast}$. By the commutative diagram (\ref{diagramMr to Crse}), $\bar{\theta}\alpha\in \pi_{7}(C_{r,1}^{5,s}) $ is a lift of $j_2^4\Sigma\nu'\in Ker(\partial_{r,1}^s)_{6\ast}$. So the short exact sequences (\ref{exact for pi7(Crs1)}) splits for $r\geq 2$, so is for $r=\infty$.
	
	For $r=1$, There is an induced map  $M_{2}^{3} \stackrel{\bar{\theta}}\hookrightarrow C_{1,1}^{5,1}=C_{1}^{5,1}$ from the left commutative diagram (\ref{diagramMr to Crse}). By Lemma 1.6. of \cite{Mukailifting}, there is an  element $\widetilde{\alpha_2}\in \pi_7(C_{1}^{5,1})$ with order 4 such that $2\widetilde{\alpha_2}=\bar{\theta}\widetilde{\eta}_3\eta_5\eta_6$ where $\widetilde{\eta}_3 \in \pi_5(M_{2}^{3})$ is a lift of $\eta_4$, i.e., $p_3\widetilde{\eta}_3=\eta_4$. We have the following commutative diagram
	\begin{align} \footnotesize{\xymatrix{
				\pi_{7}(S^5)\ar@/^1.5pc/[rr]^{\eta_{4\ast}\cong}\ar[r]^{~~~\widetilde{\eta}_{3\ast}}&  \pi_{7}(M_{2}^{3})\ar[r]^-{p_{3\ast}}\ar[d]^{\bar{\theta}_{\ast} }&Ker \partial_{6\ast}=\Z_2\{\Sigma\nu'\} \ar@{^{(}->}[d]_{j^4_{2\ast}}\\
				&\pi_{7}(C_{1}^{5,1}) \ar[r]^-{q^1_{1,1\ast}}&Ker (\partial_{1,1}^1)_{6\ast}=\Z_4\{j_1^5\eta_5\eta_6+j_2^4\Sigma\nu'\},}}\nonumber
	\end{align}
	$q^1_{1,1\ast}\bar{\theta}_{\ast}\widetilde{\eta}_{3\ast}(\eta_5\eta_6)
	=q^1_{1,1\ast}(\bar{\theta}\widetilde{\eta}_{3}\eta_5\eta_6)=2q^1_{1,1\ast}(\widetilde{\alpha_2});$
	
	On the other hand
	
	$q^1_{1,1\ast}\bar{\theta}_{\ast}\widetilde{\eta}_{3\ast}(\eta_5\eta_6)=j^4_{2\ast}\eta_{4\ast}(\eta_5\eta_6)=2(j_1^5\eta_5\eta_6+j_2^4\Sigma\nu')$
	
	Hence $q^1_{1,1\ast}(\widetilde{\alpha_2})=\pm (j_1^5\eta_5\eta_6+j_2^4\Sigma\nu')$. Thus the short exact sequences (\ref{exact for pi7(Crs1)}) splits for $r=s=1$.
	
	For $\infty\geq s\geq 2$, there is a  commutative ladder
	\begin{align} \footnotesize{\xymatrix{
				S^4\vee S^3\ar[r]^-{f_{1,1}^1}\ar@{=}[d]_{id} &S^4\vee  S^3\ar[r]\ar[d]_{d^{s-1}_{0}}& C_{1}^{5,1} \ar[r]^{q^1_{1,1}}\ar[d]^{\bar{\mu} }& S^5\vee S^4\ar@{=}[d]_{id}\\
				S^4\vee S^3 \ar[r]^{f^s_{1,1}} &  S^4\vee S^3\ar[r]&C_{1}^{5,s} \ar[r]^{q_{1,1}^s}&S^5\vee S^4,}}
		\nonumber
	\end{align}
	where $d^{s-1}_{0}=(j_12^{s-1}\iota_4,j_2\iota_3)$.
	
	Then  $q^s_{1,1\ast}(\bar{\mu}\widetilde{\alpha_2})=q^1_{1,1\ast}(\widetilde{\alpha_2})=\pm (j_1^5\eta_5\eta_6+j_2^4\Sigma\nu')$. It implies the short exact sequences (\ref{exact for pi7(Crs1)}) also splits for $r=1,  \infty\geq s\geq 2$.
\end{proof}

So\\
$\pi_{7}(C_{r,1}^{5,s})\cong \left\{
\begin{array}{ll}
	\Z_{2^{min\{s-\epsilon_r,2\}}}\oplus\Z_{2^{min\{s+1,r+1\}}}\oplus\Z_{2^{s+2}}\oplus \Z_4, & \hbox{$r\geq1$;}\\
	\Z_{2^{min\{s,2\}}}\oplus\Z_{2^{s+1}}\oplus\Z_{2^{s+2}}\oplus \Z_4\oplus\Z_{(2)}, & \hbox{$r=\infty$.}
\end{array}
\right.$

From $\pi_{7}(C_{r,1}^{5,\infty})=\pi_7(S^4\vee C_r^5)\cong \pi_7(C_r^5)\oplus \pi_7(S^4)\oplus \pi_7(C_r^8)$,
$\pi_{7}(C_{\infty,1}^{5,s})=\pi_7(S^4\vee C^{5,s})\cong \pi_7(C^{5,s})\oplus \pi_7(S^4)\oplus \pi_7(C^{8,s})$ and  $\pi_7(C_r^8)=0$, $\pi_7(C^{8,s})\cong \Z_{2^{s+1}}$ (stable) in \cite{ZP},
we get
\begin{align}
	&\pi_{7}(C_{r}^{5})\cong \Z_4\oplus \Z_{2^{r+1}}, r\geq 1\nonumber\\
	&\pi_{7}(C^{5,s})=\pi_{7}(C_{r,1}^{5,s})\cong  \Z_{2^{min\{s,2\}}}\oplus\Z_{2^{s+2}}, s\geq 1. \nonumber\\
	&\pi_{7}(C_{r}^{5,s})=\pi_{7}(C_{r,1}^{5,s})\cong
	\Z_{2^{min\{s-\epsilon_r,2\}}}\oplus\Z_{2^{min\{s+1,r+1\}}}\oplus\Z_{2^{s+2}}\oplus \Z_4, r\geq1,s\geq 1. \nonumber
\end{align}

First author was partially supported by National Natural Science Foundation of China (Grant No. 11701430); second author  were partially supported by National Natural Science Foundation of China (Grant No. 11971461).




\begin{appendix}
	
	\section{}
	\label{sec:Appendix}
	\begin{lemma}\label{lem[lota4,iota4]H(nu4)}
		\begin{align}
			&(2^r\iota_4)\nu_4=2^{2r}\nu_4-2^{r-1}(2^r-1)\Sigma \nu';\nonumber\\
			&(-2^r\iota_4)\nu_4=2^{2r}\nu_4-2^{r-1}(2^r+1)\Sigma \nu'.\nonumber
		\end{align}
	\end{lemma}
	\begin{proof}
		By Proposition 2.10 of \cite{N.Oda} and (5,8) of \cite{Toda}, $(2\iota_4)\nu_4=2\nu_4\pm \binom{2}{2}[\iota_4,\iota_4]H(\nu_4))=4\nu_4-\Sigma\nu'$ or $\Sigma\nu'$, where $H$ is the second Hilton-Hopf invariant. If $(2\iota_4)\nu_4=\Sigma \nu'$, then $(4\iota_4)\nu_4=(2\iota_4)(2\iota_4)\nu_4=(2\iota_4)(\Sigma\nu')=2\Sigma \nu'$. On the other hand, $(4\iota_4)\nu_4=4\nu_4\pm \binom{4}{2}[\iota_4,\iota_4]H(\nu_4)=16\nu_4-6\Sigma\nu'$ or $-8\nu_4+6\Sigma\nu'$, which is not equal to $2\Sigma \nu'$. Thus $(2\iota_4)\nu_4\neq \Sigma \nu'$, i.e., $(2\iota_4)\nu_4= 4\nu_4-\Sigma\nu'$.
		\begin{align}
			&(2^r\iota_4)\nu_4=(2^{r-1}\iota_4)(2\iota_4)\nu_4=(2^{r-1}\iota_4)(4\nu_4-\Sigma\nu')=4(2^{r-1}\iota_4)\nu_4-2^{r-1}\Sigma\nu' \nonumber\\
			&=2^{2r}\nu_4-2^{r-1}(2^r-1)\Sigma \nu' ~~(\text{by induction}).\nonumber
		\end{align}
		\begin{align}
			&(-2^r\iota_4)\nu_4=(-\iota_4)(2^{r}\iota_4)\nu_4=(-\iota_4)(2^{2r}\nu_4-2^{r-1}(2^r-1)\Sigma \nu' ) \nonumber\\
			&=(-\iota_4)(2^{2r}\nu_4)+2^{r-1}(2^r-1)\Sigma \nu' =-2^{2r}\nu_4\pm\binom{2}{2}[\iota_4,\iota_4]H(2^{2r}\nu_4)+2^{r-1}(2^r-1)\Sigma \nu'\nonumber\\
			&=2^{2r}\nu_4-2^{r-1}(2^r+1)\Sigma \nu'~or~-3\cdot 2^{2r}\nu_4+(3\cdot2^{2r-1}-2^{r-1})\Sigma \nu' \label{-2rnu_41}\\
			&\text{On the other hand}\nonumber\\
			&(-2^r\iota_4)\nu_4=-2^r\nu_4\pm\binom{-2^r}{2}[\iota_4,\iota_4]H(\nu_4)=-2^r\nu_4\pm\binom{2^r+1}{2}[\iota_4,\iota_4]H(\nu_4)\nonumber\\
			&=2^{2r}\nu_4-2^{r-1}(2^r+1)\Sigma \nu'~or~-(2^{2r}+2^{r+1})\nu_4+2^{r-1}(2^r+1)\Sigma\nu'.\label{-2rnu_42}
		\end{align}
		Compare (\ref{-2rnu_41}) with (\ref{-2rnu_42}), we get $(-2^r\iota_4)\nu_4=2^{2r}\nu_4-2^{r-1}(2^r+1)\Sigma \nu'$.

	\end{proof}

	\begin{proof}[\textbf{Proof of Lemma \ref{lem detemine v,w}}]
		Let $a=\widetilde{\iota_{7}}^{43}$, $b=j^s_{0}j_1\nu_4$, $c=j^s_{0}j_1\Sigma\nu'$, $d=j^s_{0}j_2\nu'\eta_6$, $e=j^s_{0}[j_1,j_2]\eta_6$.
		Let $w=2^{\alpha}w'$, where $2\nmid w'$; $v=2^{\beta}v'$, where $2\nmid w', 2\nmid v'$. By (\ref{Coker (partial0s)7ast vw}),
		\begin{align}
			&H^s\cong \frac{\Z_{(2)}\{a,b,c,d,e\}}{L'_s}\nonumber\\
			&L'_{s}=\left \langle  2^{s+1}b-2^{s}c, 2^{2s}b-2^{s\!-\!1}(2^s-1)c, 2^sa+2^{\beta}b+2^{\alpha}c+u_2d+z_2e, 2^{s+1}b, 4c,2d,2e\right \rangle \nonumber\\
			&=\left \langle  2^{s+1}b, 2^{\alpha_s}c, 2^sa+2^{\beta}b+2^{\alpha}c+u_2d+z_2e,2d,2e\right \rangle \nonumber
		\end{align}
		Note that a $\Z_{(2)}$-linear isomorphism of  $ \Z_{(2)}\{a,b,c\}$ dose not change  the group structure of $H^s$.
		
		Assume $u'=1\in \Z_2$, then
		\begin{align}
			&L'_{s}=\left \langle  2^{s+1}b, 2^{\alpha_s}c, 2^sa+2^{\beta}b+2^{\alpha}c+d+z_2e,2d,2e\right \rangle \nonumber\\
			&=\left \langle  2^{s+1}b, 2^{\alpha_s}c, 2^sa+2^{\beta}b+2^{\alpha}c+d+z_2e,2^{s+1}a+2^{\beta+1}b+2^{\alpha+1}c,2e\right \rangle \nonumber\\
			&H^s\cong \frac{\Z_{(2)}\{a,b,c,e\}}{ \left \langle  2^{s+1}b, 2^{\alpha_s}c,2^{s+1}a+2^{\beta+1}b+2^{\alpha+1}c,2e\right \rangle }\nonumber
		\end{align}
		which  has at most four (resp. three) cyclic direct summands for $s\geq 2$ (resp. $s=1$) and contradicts to (\ref{H^s}). Thus $u'=0$. By the same argument, we get $z=0$ and $\beta\geq 1$, $\alpha\geq1$ when $s\geq 2$. So we get
		
		$H^s\cong \Z_2\{d\}\oplus \Z_2\{e\}\oplus \frac{\Z_{(2)}\{a,b,c\}}{L_s}$ with
		$L_{s}=\left \langle 2^{\alpha_s}c, 2^sa+2^{\beta}b+2^{\alpha}c, 2^{s+1}b\right \rangle $.
		
		If $\alpha< \alpha_s$, then $\alpha_s=2$, $s\geq 2$,  $\alpha= 1$.
		
		$L_{s}=\left \langle 2^{s+1}a+2^{\beta+1}b, 2^sa+2^{\beta}b+2c, 2^{s+1}b\right \rangle$.
		
		$\Z_{2^{s+1}}\oplus \Z_{4}\oplus\Z_{2^s}\oplus \Z_2\oplus \Z_2\cong H^s\cong \Z_2\oplus\Z_2\oplus \Z_2 \oplus A $ for some group $A$, which is impossible.
		
		Hence $\alpha\geq \alpha_s$, i.e., $2^{\alpha_s}\mid w$. $L_{s}=\left \langle 2^{\alpha_s}c, 2^sa+2^{\beta}b, 2^{s+1}b\right \rangle $.
		
		If $\beta<s$, $L_{s}=\left \langle 2^{\alpha_s}c, 2^sa+2^{\beta}b, 2^{2s+1-\beta}a\right \rangle $.
		$H^s\cong \Z_2\oplus\Z_2\oplus \Z_{\alpha_s}\oplus A$ with exact sequence $0\rightarrow \Z_{2^{2s+1-\beta}}\rightarrow A\rightarrow \Z_{2^{\beta}}\rightarrow 0$. This is a contradiction since $A\cong \Z_{2^{s}}\oplus \Z_{2^{s+1}}$ is not a solution of above exact sequence.
		So $\beta\geq s$.  We complete the proof of Lemma \ref{lem detemine v,w}.
	\end{proof}
	
	\begin{proof}[\textbf{Proof of Lemma \ref{lem Cokerpartial rs1}}]
		$Coker(\gamma_{r,1}^{s})_{7\ast}= \Z_{2^{s+2}}\{j_1\nu_4\}\oplus \Z_4\{j_1\Sigma\nu'\}\oplus \Z_2\{j_2\nu'\eta_6\}$ in (\ref{Cokergammarse}).
		
		Simplify $a:=\widetilde{\rho_{r,1}^s}$, $b:=j_{S^6} \eta_6$, $c:=j_{r,1}^sj_1\nu_4$, $d:=j_{r,1}^sj_1\Sigma\nu'$, $e:=j_{r,1}^sj_2\nu'\eta_6$. From (\ref{partial7 j1nu5}),(\ref{partial7j^4Sigmanu'eta_7}),(\ref{partial7 j2^4nu4eta7}),
		(\ref{partial7r1s [j15j24]}), Lemma \ref{lem detemine x,y,t'} and Lemma \ref{lem k l}, we get
		\begin{align}
			&Coker(\partial_{r,1}^s)_{7\ast}=\frac{\Z_{(2)}\{a,b,c,d,e\}}{I_r^s}\nonumber.
		\end{align}
		\begin{align}
			&I_r^s\!\!=\!\!\left \langle 2b,2^{s\!+\!2}c,4d, 2e, 2^{2s}c\!-\!2^{s\!-\!1}(2^s\!\!-\!\!1)d\!+\!e, 2^{min\{r,\!s\}}ta\!+\!b\!+\!2^{min\{r,\!s\}}k'c\!+\!2^{min\{s\!-\!1,1\}}l'd\!+\!ue, \epsilon_r e\right \rangle  \nonumber\\
			&=\!\!\langle 2^{min\{r+1,\!s+1\}}(ta\!+\!k'c)\!+2^{min\{s,2\}}l'd,2^{s\!+\!2}c,4d, 2e, 2^{2s}c\!-\!2^{s\!-\!1}(2^s\!\!-\!\!1)d\!+\!e,   \nonumber\\ &~~~~~~~~2^{min\{r,\!s\}}ta\!+\!b\!+\!2^{min\{r,\!s\}}k'c\!+\!2^{min\{s\!-\!1,1\}}l'd\!+\!ue, \epsilon_r e \rangle  \nonumber\\
			&Coker(\partial_{r,1}^s)_{7\ast}=\frac{\Z_{(2)}\{a,c,d,e\}}{I^{'s}_r} \nonumber\\
			&I^{'s}_r=\!\!\left \langle  2^{min\{r+1,\!s+1\}}(ta\!+k'c)\!+\!2^{min\{s,2\}}l'd,2^{s\!+\!2}c,4d, 2e, 2^{2s}c\!-\!2^{s\!-\!1}(2^s\!\!-\!\!1)d\!+\!e, \epsilon_r e\right \rangle  \nonumber\\
			&\text{since}~2^{min\{s,2\}}l'd\in \left \langle 2^{s\!+\!2}c,4d, 2e, 2^{2s}c\!-\!2^{s\!-\!1}(2^s\!\!-\!\!1)d\!+\!e\right \rangle, \nonumber\\
			&I^{'s}_r=\!\!\left \langle  2^{min\{r+1,\!s+1\}}(ta\!+\!k'c),2^{s\!+\!2}c,4d, 2e, 2^{2s}c\!-\!2^{s\!-\!1}(2^s\!\!-\!\!1)d\!+\!e, \epsilon_r e\right \rangle  \nonumber\\
			&Coker(\partial_{r,1}^s)_{7\ast} =\Z_{2^{min\{r\!+\!1,\!s+\!1\}}}\{a+\frac{k'}{t}c\}\oplus \frac{\Z_{(2)}\{c,d,e\}}{I^{''s}_r} \nonumber\\
			&I^{''s}_r=\left \langle 2^{s\!+\!2}c,4d, 2e, 2^{2s}c\!-\!2^{s\!-\!1}(2^s\!\!-\!\!1)d\!+\!e, \epsilon_r e\right \rangle\nonumber.
		\end{align}
		For $\infty\geq r\geq 2$
		\begin{align}
			&I^{''s}_r=\left \langle 2^{s\!+\!2}c,4d, 2^{2s\!+\!1}c\!-\!2^{s}(2^s\!\!-\!\!1)d, 2^{2s}c\!-\!2^{s\!-\!1}(2^s\!\!-\!\!1)d\!+\!e\right \rangle\nonumber\\
			&=\left \langle 2^{s\!+\!2}c,2^{min\{s,2\}}d, 2^{2s}c\!-\!2^{s\!-\!1}(2^s\!\!-\!\!1)d\!+\!e\right \rangle\nonumber\\
			&Coker(\partial_{r,1}^s)_{7\ast} \cong\Z_{2^{min\{r\!+\!1,\!s+\!1\}}}\oplus \frac{\Z_{(2)}\{c,d\}}{\left \langle 2^{s\!+\!2}c,2^{min\{s,2\}}d\right \rangle } \nonumber\\
			&\cong \Z_{2^{min\{r\!+\!1,\!s+\!1\}}}\oplus \Z_{2^{min\{s-\epsilon_r,2\}}}\oplus \Z_{2^{s+2}}.\nonumber
		\end{align}
		For $r=1$,
		\begin{align}
			&I^{''s}_1=\!\!\left \langle  2^{s\!+\!2}c,4d, 2^{2s}c\!-\!2^{s\!-\!1}(2^s\!\!-\!\!1)d, e\right \rangle
			=\left\{\!\!\!
			\begin{array}{ll}
				\left \langle  8c, 4c\!-\!d, e\right \rangle , & \hbox{$s=1$;} \\
				\left \langle  2^{s\!+\!2}c,2^{min\{2,s-1\}}d, e\right \rangle , & \hbox{$\infty\geq s\geq 2$.}
			\end{array}
			\right.
			\nonumber\\
			&\Rightarrow  Coker(\partial_{1,1}^s)_{7\ast}\cong
			\Z_4\oplus \Z_{2^{s+2}}\oplus \Z_{2^{min\{2,s-1\}}}, \infty\geq s\geq 1.\nonumber
		\end{align}
		
		We complete the proof of Lemma \ref{lem Cokerpartial rs1}.
	\end{proof}

\end{appendix}


\begin{thebibliography}{99}
	
	
	\bibitem{BauJHC} Baues H J.  On homotopy classification problems of J.H.C. Whitehead. Lect. Notes in Math., 1985, 1172: 17--55
	\bibitem{Baues Hotpy Homolgy}
	Baues H J. Homotopy type and homology. Oxford University Press,1996
	\bibitem{RefChang}
	Chang S C. Homotopy invariants and continuous mappings. Proc. Roy. Soc. London. Ser.A, 1950, 202: 253--263
	\bibitem{Cohen}
	Cohen J M. Stable homotopy. Lect. Notes Math, 1970, 165
	\bibitem{C.Costoya}
	Costoya C, $\acute{M}$endez D, Viruel A. The group of self-homotopy equivalences of $\mathbf{A}_n^2$-polyhedra. Journal of Group Theory, 2020, 23.4: 575--591
	\bibitem{Drozd}
	Drozd Y A. On classification of torsion free polyhedra. Preprint series, MaxPlanckInstitut f\"{u}r Mathematik (Bonn) 2005,92
	\bibitem{Gray}
	Gray B. On the homotopy groups of mapping cones. Proc. London Math. Soc, 1973, 26.3: 497--520
	\bibitem{Hilton1950}
	Hilton P J. Calculation of the homotopy groups of $\mathbf{A}_n^2$-polyhedra (i). Quarterly Journal of Mathematics, 1950, 1.1: 299--309
	\bibitem{Hilton1951}
	Hilton P J. Calculation of the homotopy groups of $\mathbf{A}_n^2$-polyhedra (ii). Quarterly Journal of Mathematics, 1951, 2: 228--240
	\bibitem{Hiltonbook}
	Hilton P J.  An Introduction To Homotopy Theory. Cambridge University Press, 1953
	\bibitem{HuangRuiWu}
	Huang R Z, Wu J. Exponential growth of homotopy groups of suspended finite complexes. Math. Z, 2020, 295: 1301--1321
	\bibitem{X.G.Liu}
	Liu X G. On the Moore Space $P^n(8)$ and Its Homotopy Groups. Chinese Annals of Math, 2007, 28A.3: 305--318
	\bibitem{D.Mendez}
	$\acute{M}$endez D. The ring of stable homotopy classes of self-maps of -polyhedra. Topology and its Applications, 2021, 290.1
	\bibitem{Mukailifting}
	Morisugi K, Mukai J. Lifting to mod $2$ Moore spaces. Journal of the Mathematical Society of Japan, 2000, 52: 515--534
	\bibitem{Mukai:pi(CPn)}
	Mukai J. The $S^1$-transfer map and homotopy groups of suspended complex projective spaces. Mathematical Journal of Okayama University, 1982, 24.2: 179--200
	\bibitem{Mukai}
	Mukai J, Shinpo T. Some homotopy groups of the mod 4 moore space. J. Fac. Sci. Shinshu Univ. 1999, 34.1: 1--14
	\bibitem{N.Oda}
	Oda N. Unstable homotopy groups of spheres. The Buletin of she Instiula for Adoaxced
	Research of Fukuoka Uaibersity, 1979, 44: 49--151
	\bibitem{Toda}
	Toda H. Composition methods in homotopy groups of spheres. Princeton University Press, 1963
	\bibitem{WJProjplane}
	Wu J. Homotopy theory of the suspensions of the projective plane. Memoirs  AMS, 2003, 162: no.769
	\bibitem{WhiteEHT}
	Whitehead G W.  Elements of Homotopy Theory, Springer-Verlag, 1978
	\bibitem{JXYang}
	Yang J X, Mukai J,  Wu J.  On the Homotopy Groups of the Suspended Quaternionic
	Projective Plane. Preprint
	\bibitem{ZP}
	Zhu Z J, Pan J Z.  The decomposability of smash product of $\mathbf{A}_n^2$
	complexes. Homology Homotopy and Applications, 2017, 19: 293--318
	\bibitem{ZLP}
	Zhu Z J, Li P C, Pan J Z.  Periodic problem on homotopy groups of Chang complexes $C_{r}^{n+2,r}$.  Homology Homotopy and Applications, 2019, 21.2: 363--375
	\bibitem{ZPhyperbolicity}
	Zhu Z J, Pan J Z. The local hyperbolicity of $\mathbf{A}_n^2$-complexes.  Homology Homotopy and Applications 2021, 23.1: 367--386
	
	
\end{thebibliography}
\end{document}